\documentclass{amsart}
\usepackage[utf8]{inputenc} 
\usepackage[T1]{fontenc}
\usepackage{lmodern}
\usepackage{graphicx}
\usepackage[english]{babel}
\usepackage{wrapfig}
\usepackage{amsthm}
\usepackage{setspace}
\usepackage{amsmath}

\usepackage{amssymb}

\usepackage{comment}
\usepackage{caption}
\usepackage[a4paper]{geometry}  

\usepackage[bottom]{footmisc}
\usepackage[all]{xy}
\usepackage{float} 
\usepackage{hyperref}
\usepackage{setspace}
\newsavebox{\auteurbm}
  {\small\slshape%
  \savebox{\auteurbm}{\upshape\sffamily#1}%
  \begin{flushleft}}
  {\\[4pt]\usebox{\auteurbm}
  \end{flushleft}\normalsize\upshape}
  
  \usepackage{mathtools}
  \usepackage{dsfont}
  \usepackage{amsmath}
  \usepackage{amsthm}

  \theoremstyle{definition}
  \newtheorem{defn}{Definition}[section]
  
  \newtheorem*{defn*}{Definition}
 
  \theoremstyle{plain}
  \newtheorem{thm}{Theorem}[section]
  
  \newtheorem{prop}{Proposition}[section]
  \newtheorem*{prop*}{Proposition}
  \newtheorem{lem}{Lemma}[section]
  \newtheorem{cor}{Corollary}[section]
   \newtheorem*{cor*}{Corollary}
  \newtheorem*{theo*}{Theorem}
  \newtheorem*{thm*}{Theorem}
  
  \theoremstyle{remark}
  
  \newtheorem{rem}{Remark}[section]

  \newtheorem{nota}{Notation}[section]

 \usepackage[toc,page]{appendix}
\usepackage{calrsfs}
\usepackage{tikz-cd}

\newcommand{\R}{\mathbb{R}}
\newcommand{\A}{\mathcal{A}}
\newcommand{\B}{\mathcal{B}}
\newcommand{\C}{\mathcal{C}}

\DeclareMathOperator*{\F}{U}

\newcommand{\U}{\operatorname{\mathcal{U}}}

\usepackage{pstricks,pst-node,pst-tree}

\usepackage{bm}

\usepackage{relsize}

 \usepackage{xcolor}

 \newcommand\N{\mathbb{N}}
 \newcommand\cat[1]{\textbf{#1}}

\newcommand{\Pa}{\mathcal{P}}

\DeclareMathOperator*{\vect}{\textbf{Gr}}

\newcommand{\colim}{\operatorname{colim}}

\usepackage{appendix}

\title{Interaction decomposition for Hilbert spaces}

\author{Grégoire Sergeant-Perthuis}

\usepackage{pdfpages} 

\usepackage[intoc]{nomencl}
\makenomenclature

\usepackage{pdfpages}

\newcommand{\Split}{\cat{Split}}
\def\Plus{\small\texttt{+}}

\DeclareMathOperator{\core}{core}

\DeclareMathOperator{\pig}{\overset{ \underline{\perp}}{\bigoplus}}
\DeclareMathOperator{\com}{\mathcal{C}}

\DeclareMathOperator{\coeq}{\operatorname{Coeq}}
\newcommand{\f}{V} 
\newcommand{\ve}{\cat{Vect}}
\newcommand{\ph}{\cat{Philb}}
\newcommand{\h}{\cat{Hilb}}
\newcommand{\ih}{\cat{IHilb}}
\newcommand{\ip}{\cat{IPhilb}}

\newcommand{\Sub}{\text{Sub}}
\newcommand{\hG}{\hat{G}}
\newcommand{\hF}{\hat{F}} 
\newcommand{\Z}{\mathbb{Z}}

\newcommand{\im}{\operatorname{im}}

\newcommand{\id}{\operatorname{id}}

\makeatletter
\newtheorem*{rep@theorem}{\rep@title}
\newcommand{\newreptheorem}[2]{%
\newenvironment{rep#1}[1]{%
 \def\rep@title{#2 \ref{##1}}%
 \begin{rep@theorem}}%
 {\end{rep@theorem}}}
\makeatother

  \theoremstyle{plain}
\newreptheorem{thm}{Theorem}
\newreptheorem{prop}{Proposition}
\newreptheorem{cor}{Corollary}

\makeatletter
\newcommand{\mylabel}[2]{#2\def\@currentlabel{#2}\label{#1}}
\makeatother

\numberwithin{equation}{section}

\newcommand{\Vect}{\cat{Vect}}

\usepackage{multiaudience}
\SetNewAudience{long}

\usepackage{minitoc}  
 
\usepackage{changepage} 
 
\usepackage{chngcntr}
\usepackage{etoolbox}
 \makeatletter
   {\par}
\makeatother

\AtBeginEnvironment{subappendices}{%
\section*{Appendix}
\addcontentsline{toc}{section}{Appendices}
\counterwithin{figure}{section}
\counterwithin{table}{section}
}

\usepackage{cite}

\begin{document}

\begin{abstract}

The decomposition into interaction subspaces is an important result for graphical models \cite{Speed}, \cite{Lauritzen} and plays a central role for results on the linearized marginal problem \cite{Kellerer1964}; similarly the Chaos decomposition plays an important role in statistical physics at the thermodynamic limit \cite{Sinai} and in probability theory in general \cite{T.Levy}. We unify and extend both constructions by defining and characterizing decomposable functors from a well-founded poset to the category of Hilbert spaces, completing previous work on decomposable collections of vector subspaces \cite{GS2} and presheaves \cite{GS3}.

\noindent \textbf{MSC2020 subject classifications:} Primary 46M15
; secondary 06F25 \\
\noindent \textbf{Keywords.} Decomposition into interaction subspaces, Functors in Hilbert spaces, Chaos decomposition.

\end{abstract}

 \maketitle

\maketitle

\section{Introduction}

Similar constructions appear in statistical physics and for Graphical models \cite{Lauritzen1}, namely what we shall call an interaction decomposition and this is because both are interested in the same objects, Gibbs states. Indeed, in order to capture the interactions between a finite number of random variables, $(X_i\in E_i, i\in I)$, when modelizing a phenomenon, one can introduce the notion of potential.

\begin{defn}[Potential and Gibbs state]
Let $(E_i,i\in I)$ be a collection of finite sets over a finite set $I$, let for $a\subseteq I$, $E_a=\prod_{i\in a}E_i$ and let $p_a:E_I \to E_a$ denote the projection onto $E_a$. Let $V(a)$ be the space of $E_a$ cylindric functions, i.e. functions $f$ in $\R^{E_I}$ for which there is $\tilde{f}\in \R^{E_a}$ such that $f=\tilde{f}\pi_a$; this space is also called the $a-$factor subspace and can be identified to $\R^{E_a}$. 

A potential $\Phi=(\phi_a\in V(a),a \subseteq I)$ is a collection of cylindric functions. We will denote the power set of $I$ as $\Pa(I)$.\\

One can associate to any potential a probability law as follows,

\begin{equation}
P=\frac{e^{\sum_{a\subseteq I} \phi_a}}{\sum_{x\in E_I}e^{\sum_{a\subseteq I} \phi_a(x)}}
\end{equation}

We shall refer to such a potential as a Gibbs state with respect to the potential $\Phi$.\\
\end{defn}

Let us note for any $\A\subseteq \Pa(I)$,  

\begin{equation}
H_\A=\{\sum_{a\in\A} \phi_a \vert \ \forall a\in \A, \ \phi_a\in V(a)\}
\end{equation}

These are called hierarchical model subspaces by Lauritzen in this reference book on Graphical Models \cite{Lauritzen} and the associated probability laws live in factorisation spaces, i.e. the space of positive functions that factorize according to $\A$, denoted as $G_\A$ for $\A\subseteq \Pa(I)$.\\

\subsection{Interaction decomposition for Graphical models}

One can decompose the space of all random variables into orthogonal bits from which one can rebuild the hierarchical model subspaces. Let $I$ be a finite set and let $E=\prod_{i\in I} E_i$ be a finite set; let us consider the canonical scalar product on $\R^{E_I}$, i.e. for any $f,g\in \R^{E_I}$, 

\begin{equation}
\langle f,g\rangle= \sum_{x\in E_i} f(x)g(x)
\end{equation}

Then the following result holds,

\begin{thm*}[Decomposition into interaction subspaces for hierarchical model subspaces]\label{i:interaction-decomposition}
The space of random variables, $\R^{E_I}$, can be decomposed into an othogonal direct sum of vector subspaces, $(S_a,a\subseteq I)$,

\begin{equation}
\R^{E_I}= \bigoplus_{a\subseteq I} S_a
\end{equation}

such that for any  lower set $\A\subseteq \Pa(I)$,

\begin{equation}
H_\A= \bigoplus_{a\in \A}S_a
\end{equation}

\end{thm*}

For $a\subseteq I$, let us note $s_a$ the orthogonal projections onto $S_a$.\\

The Decomposition into interaction subspaces for hierarchical model subspaces is very useful when one wants to test if a (strictly positive) probability distribution is in a factorisation space, i.e. for $\A\subseteq \Pa(I)$,

\begin{equation}
P\in G_\A \ \iff\   \forall a\notin \A, \ s_a(\ln P)=0
\end{equation}

\subsection{Chaos decomposition in statistical physics}

A simpler, and older case of the interaction decomposition is the chaos decomposition for a $\N$-filtered collection of subspaces of the space of all random variables, $(H_n,n\in \N)$, with $H_n\subseteq H_m$ for $n\leq m$. The space of random variables is $L^2(\R^I,P)$ where $P$ is a Gaussian mesure and $I=\Z^d$. Let $n\in \N$, it represents a degree, let for any integer $m\leq n$, any $x:[1,m]\to I$, and any $\phi\in \R^I$,

\begin{equation}
\Psi(x)(\phi)=\prod_{i\in [1,n]} \phi(x_i)
\end{equation}

 $H_n$ is the Hilbert subspace generated by the $\Psi(x)$ of degree at least $n$ and indeed $H_n \subseteq H_{n+1}$.\\

The construction of a decomposition of the $(H_n,n\in \N)$ relies on the Hermite-Ito polynomials for Gaussian fields  (see Section \ref{chapitre-4-chaos-decomposition}). The aim of this construction is to define the space of observables, of potentials, when the number of random variables $I$ is not finite. It is an important construction in the field of statistical physics at the thermodynamic limit \cite{Sinai}.\\

\subsection{Aim and previous work}

In previous work \cite{GS2}, \cite{GS3} we have generalized the decomposition into interaction subspaces in two directions. The first one was to characterize collections of vector subspaces of a given vector space $V$ over a poset $\A$, $(\F(a)\subseteq V, a\in \A)$, such that there is a collection of vector subspaces of $(S_a,a\in \A)$ for which,

\begin{equation}
\forall a\in \A, \quad \F(a)= \bigoplus_{b\leq a} S_a
\end{equation}

The direct sum is not necessarily an orthogonal one. Such collection, $(\F(a),a\in \A)$, is called a decomposable collection of vector subspaces of $V$  (Definition 3.1 \cite{GS2}) and leads to the definition of decomposable functors. Let us denote $\A^{\Plus}=\A\oplus 1$ the poset that is the sum of $\A$ and the one element poset $1$ defined as the reunion of elements of $\A$ and $1$ with $1$ being final in $\A^{\Plus}$ and for any $a,b\in \A$,

\begin{equation}
b\leq_{\A^{\Plus}} a \implies b\leq_{\A} a
\end{equation}

$\A^{\Plus}$ is $\A$ with the addition of a maximal element. When $(\F(a),a\in \A)$ is decomposable so is $(V, \F(a), a\in \A)$ and there is $S_1$ such that,

\begin{equation}
    \F(1)= V= \bigoplus_{a\in \A^{\Plus}} S_a
\end{equation}

 There are in fact several possible decompositions, i.e. collections of subspaces $(S_a,a\in \A)$, for a given collection of vector subspaces $(\F(a),a\in \A)$; by considering the additional data of a collection of projectors $(\pi_a,a\in \A)$ onto the $(\F(a),a\in \A)$, in other words such that $\pi_a: V\to V$ is a projector onto $\F(a)$, one can distinguish the different decompositions. In this case one can define what a decomposable collection of projectors is (Definition 2.4 \cite{GS3}) by asking that, for $a\in \A$, $\pi_a$ is the projections from $\bigoplus_{b\in \A^{\Plus}} S_b$ to $\bigoplus_{b\leq a} S_b$. This notion of decomposable collection of projectors can be extended to the notion of decomposable collection of presheaves (Definition 4.3 \cite{GS3}).
 
 In both cases the characterization of decomposable objects relies on properties that follow the same spirit but are different in their implementation (Definition 3.3 \cite{GS2}, 4.7 \cite{GS3}), this is why we decided to keep the same denomination, `intersection properties', for both properties. The intersection property for collections of vector subspaces has an interesting interpretation when restricted to the case of factor spaces, as it is simply an extension of the Bayesian intersection property (see \cite{GS1}) and decomposable objects appear when considering classical problems in the more general setting of the study of probability theory with a categorical flavour, initiated by Giry \cite{Giry} and Lawvere \cite{Lawvere}. For example results of Kellerer \cite{Kellerer1964} for the marginal problem can be extended when considering the generalization of the linearized marginal problem thanks to decomposable functors \cite{GSDB} and the generalized version of Gibbs states for diagrams in the category of Markov kernels are completely characterized for decomposable presheaves \cite{GSthese}.

In this document we want to extend one step further the link between decomposability and a notion of intersection property. We want to unify both of the constructions introduced in the previous subsection into a unique algebraic framework and study when they occur in greater generality. Both of these constructions are a decomposition of a collection of Hilbert subspaces of a given Hilbert space $H$, indexed on a partially ordered set. In greater generality these collections are functors from a poset to the category of Hilbert spaces with isometries as morphisms. Not all these functors admit a decomposition, the aim of this document is to prove that the ones that do are exactly those that satisfy an intersection property.

\pagebreak


\subsection{Main results of this document}

\subsubsection{Prerequisites}

In order to state the main results of this document we need firstly to recall some concepts.

\begin{defn}[Categories]

We shall consider several categories; the category that has as objects Hilbert spaces and as morphism continuous linear applications will be noted $\cat{Hilb}$, when the morphisms are isometries we will call it $\cat{IHilb}$; when the objects are pre-Hilbert spaces and the morphisms are continuous linear applications the corresponding category is $\cat{PHilb}$; finally the category that has as object pre-Hilbert spaces and as morphisms isometries shall be denoted as $\cat{IPhilb}$ and the category that has as objects vector spaces and as morphisms linear applications shall be noted as $\cat{Vect}$. The reference to the field $\mathbb{K}= \R$ or $\mathbb{C}$ is made implicit.
\end{defn}

\begin{repprop}{a-h-c:completion}[Completion]
Let $H$ be a pre-Hilbert space, there is a Hilbert space that we shall note as $\com H$ and a continuous injective linear application $\eta_H:H\to \com H$ such that for any Hilbert space $H_1$ and any continuous linear application (of $\cat{PHilb}$), $\phi: H\to H_1$, there is a unique continuous linear application $\com \phi:\com H\to H_1$ such that $\com \phi \eta_H= \phi$. $\com H$ is called the completion of $H$.

\end{repprop}

\begin{nota}
Let $\A$ be a poset, and let $\B\subseteq \A$ be a subposet of $\A$. We will denote,

\begin{equation}
\hat{\B}= \{a\in \A \quad \exists b\in \B, a\leq b \}
\end{equation}

the lower-set associated to $\B$ and any subposet of $\A$ such that $\hat{\B}= \B$ will be call a lower-set. The set of lower-sets of $\A$ will be denotes $\U(\A)$.

\end{nota}

\begin{defn}[Poset extension]
Let $\A_2$ be the subposet of $\A\times \U(\A)$ constituted of couples $(\alpha,\B)$ such that $\B\subseteq \hat{\alpha}$.
\end{defn}

\begin{repcor}{a-h-c:join-lattice-isometry}[$\Sub(H)$ as join semi-lattice]
For the category $\ih$, and for a given Hilbert space $H$, the "poset" $\Sub(H)$ of subobjects of $H$ has a join, $\bigvee_{i\in I}H_i$, for any collection, $(H_i,i\in I)$, of objects of $\Sub(H)$.
\end{repcor}

\begin{defn}

If every strictly decreasing sequence of elements of a poset $\A$ terminates, the poset is said to be well-founded.

\end{defn}

\subsubsection{Definition of decomposability}

\begin{defn}[Core of a category]
Let $\cat{C}$ be any category the subcategory of $\cat{C}$ that has the same objects than $\cat{C}$ but which morphisms are the isomorphisms of $\cat{C}$ is the core of $\cat{C}$ and denoted as $\core\cat{C}$. 

\end{defn}

\begin{defn*}[Decomposability for funtors in $\ih$]
Let $\A$ be any poset, a functor $G:\A\to\cat{IHilb}$ is decomposable if there is a collection of functors $(S_a:\A\to\core\ih,a\in \A)$ such that, 
\begin{equation}
G\cong \com \underset{a\in \A}{\bigoplus}S_a1[a\leq.]
\end{equation}

i.e. there is a natural transformation, $(\phi_a,a\in \A)$, where for any $a\in \A$, $\phi_a$ is an isometric isomorphisms from $G(a)$ to $\com \bigoplus_{b\leq a}S_b(a)$.
\end{defn*}

\subsubsection{Definition of the intersection property}
\begin{defn}[Meet semi-lattice]
Let $\A$ be a poset, $\A$ is a meet semi-lattice when for any  $a,b\in\A$, there is $d$ such that,

$$\forall c\in \A,\quad  c\leq a \quad \& \quad c\leq b \implies c\leq d$$

$d$ is unique and we shall note it $a\wedge b$.
\end{defn}

\begin{defn*}[Intersection property]
Let $\A$ be any poset and $G$ be a functor from $\A$ to $\ih$. For any $(\alpha,\B) \in \A_2$ let us note $\pi^{\alpha}(\B)$ the orthogonal projection of $G(\alpha)$ onto $G^{\B}_\alpha= \bigvee_{a\in \B}G^a_\alpha$, i.e.,
\begin{equation}
\pi^{\alpha}(\B)=\bigvee_{a\in \B}G^a_\alpha \left(\bigvee_{a\in \B}G^a_\alpha\right)^\dagger
\end{equation}

We shall say that $G$ satisfies the intersection property if

\begin{equation}\tag{I'}
\forall (\alpha,a),(\alpha,b)\in \A_1,\quad \pi^{\alpha}(\hat{a}\cap \hat{b})=\pi^{\alpha}(\hat{a})\pi^{\alpha}(\hat{b})
\end{equation} 

\end{defn*}

When $\A$ is a meet semi-lattice, the intersection property can be restated as,

\begin{equation}
\pi^{\alpha}(\widehat{a\wedge b})=\pi^{\alpha}(\hat{a})\pi^{\alpha}(\hat{b})
\end{equation}

\subsubsection{Theorem}

The main result of this document is the equivalence between the intersection property and the interaction decomposition.

\begin{repthm}{d-h-2:thm}[Main theorem]
Let $\A$ be a well founded poset and $G$ be a functor from $\A$ to $\ih$, $G$ is decomposable if and only if it statifies the intersection property.

\end{repthm}

In order to prove this Theorem we must first prove it in the particular case of an increasing collection of Hilbert subspace of a given Hilbert space. In this context the definition of interaction decomposition is simpler. \\

Let $H$ be a Hilbert space and $(K_i,i\in I)$ be an orthogonal collection of Hilbert subspaces of $H$ over any set $I$; we shall note the closure of the sum of such collection of subspaces as,

\begin{equation}
\overline{\sum_{i\in I}K_i}= \underset{i\in I}{\pig} K_i
\end{equation}

\begin{defn*}[Decomposable collection of Hilbert subspaces]
Let $\A$ be any poset, let $H$ be a Hilbert space and let $(H_a \subseteq H,a\in \A)$ be a collection of Hilbert subspaces of $H$. $(H_a,a\in\A)$ is said to be decomposable if there is a collection of Hilbert subspaces of $H$, $(S_a\subseteq H,a\in \A^{\Plus})$, such that for any $a\in \A$,

\begin{equation}
H=\underset{a\in \A^{\Plus}}{\pig} S_a
\end{equation}

and for any $a\in \A$,

\begin{equation}
H_a=\underset{b\leq a}{\pig} S_b
\end{equation}

We shall call $(S_a,a\in \A^{\Plus})$ a decomposition of $(H_a,a\in \A)$.

\end{defn*}

The simpler version of the Theorem \ref{d-h-2:thm} is as follows.
\begin{repthm}{d-h:thm}
Let $\A$ be a well-founded poset, let $H$ be a Hilbert space, let $(H_a\in \vect{H}, a\in \A)$ be a collection of subspaces of $H$; $(H_a\in \vect{H}, a\in \A)$ is decomposable if and only if it satisfies the intersection property.
\end{repthm}

Which can be stated, when $\A$ is a meet semi-lattice, as follows.

\begin{repcor}{chapitre-4-decomposition-hilbert-cor-simple}
Let $\A$ be a well-founded meet semilattice. Let $H$ be a Hilbert space and $(H_a\in \vect{H},a\in \A)$ be a collection of subspace of $H$ such that for any $a,b\in \A$,

\begin{equation}
\pi_a\pi_b=\pi_{a\wedge b}
\end{equation}

where $\pi_a$ is the orthogonal projection of $H$ onto $H_a$ for any $a\in \A$. Then, $(H_a,a\in \A)$ is decomposable.
\end{repcor}

Furthermore any decomposable functor is in fact isometrically isomorphic to an increasing collection of Hilbert subspaces of a given Hilbert space.

\begin{repprop}{d-h-2:thm-lem}
Let $\A$ be any poset and $G$ be a functor from $\A$ to $\ih$. If $G$ is decomposable then there is $H$ a Hilbert space and a functor $\U:\A\to \vect H$ such that $G\cong \U$.

\end{repprop}
\subsection{Structure of this document}

We start by recalling what the chaos decomposition (Section \ref{chapitre-4-chaos-decomposition}). This motivates our study of decomposable collections of Hilbert subspaces (Section \ref{chapitre-4-decomposition-collection-subspaces}). In Section \ref{d-h:Characterizing the decompositions} we then remark that there is a unique decomposition for decomposable collections of subspaces, then we define the intersection property in this context (Section \ref{d-h:Intersection property}) and prove the equivalence theorem for collections of Hilbert subspaces (Section \ref{d-h:Theorem: equivalence between decomposition and intersection property}).\\

In the second part, Section \ref{chapitre-4-extension}, of this document we explain how decomposability for collections of Hilbert subspaces translates in the context of functors to $\ih$ (Section \ref{d-h-2:Decomposability}). In Section \ref{d-h-2:No easy road} we discuss why the result of the first sections are not enough to prove the general equivalence between theorem and the rest of Section \ref{chapitre-4-extension} is dedicated to the proof of the equivalence the intersection property and decomposability for functors from a well-founded poset to $\ih$.\\

Finally in Section \ref{bootstrap} we explain how the interaction decomposition for presheaves \cite{GS3} and the one we propose here are related.

\section{Interaction decomposition and Chaos decomposition}\label{chapitre-4-chaos-decomposition}

Let us first present some facts for the decomposition into interactions subspaces for factor spaces.

\begin{defn}
Let $H$ be a Hilbert space, we shall call the Grassmannian of $H$ in $\h$, noted $\vect_{\h}{H}$ or simply $\vect{H}$, the set of all Hilbert subspaces of $H$, i.e. closed vector subspaces of $H$. Respectively $\vect_{\cat{Vect}}{H}$ for the vector subspaces of $H$.
\end{defn}

\begin{nota}
We will denote the set of increasing functions from a poset $\A$ to $\vect{H}$ as $[\A, \vect{H}]$.
\end{nota}

\begin{defn}
Let $H$ be a Hilbert space, let $\A$ be any poset, let $(W_a\in \vect{H} ,a\in \A)$ be a collection of subspaces of $H$; if for any $a,b\in \A$ such that $a\neq b$, $W_a,W_b$ are othogonal, then we shall note the closure of $\sum_{a\in \A}W_a$ as $\pig_{a\in \A}W_a$. 
\end{defn}

One can define the sum of a collection of pre-Hilbert space but it is not universal (see Appendix B Definition \ref{a-h-c:direct-sum}, Proposition \ref{a-h-c:no-colimit}). Let us recall that $\com: \ph\to \h$ is the completion functor that is left adjoint to the forgetful functor $U$ (Appendix B Proposition \ref{a-h-c:completion}, \ref{a-h-c:adjoint})\\

Furthermore for the category $\ih$ and for a given Hilbert space $H$, the "poset" $\Sub(H)$ of subobject of $H$ has a join for any set $I$ of objects of $\Sub(H)$ (Appendix B Corollary \ref{a-h-c:join-lattice-isometry}). Given a collection $(\phi_i,i\in I)$ of monomorphism that have as codomain $H$, $\bigvee_{i\in I}\phi_i$ can be represented as $\overline{\sum_{i\in I}\im \phi_i}\hookrightarrow  H$, where $\im\phi_i$ is the image of $\phi_i$ in $\ve$ but is also a Hilbert space as $\phi_i$ is an isometry for any $i\in I$. In particular for collections of Hilbert subspace $(H_i\subseteq H,i\in I)$, a representant of $\bigvee_{i\in I}H_i$ is $\overline{\sum_{i\in I}H_i}$, where $H_i$ is implicitly identified to its inclusion in $H$.

\begin{rem}
Let $\A$ be any poset and let $(W_a\in \vect{H},a\in \A)$ be a collection of subspaces of a Hilbert space $H$ such that $\overline{\sum_{a\in \A}W_a}=\pig_{a\in \A}W_a$ then 
\begin{equation}
\com \bigoplus_{a\in \A} W_a\cong \pig_{a\in \A}W_a=\bigvee_{a\in \A}W_a 
\end{equation}
\end{rem}

\begin{thm}[Decomposition into interaction subspaces]\label{interaction-decomposition-litterature}
Let $I$ be a finite set and $(H_a,a\subseteq I)$ be the collection of factor spaces; for any $a\subseteq I$, let $S_a=H_a\cap \left(\underset{b\subsetneq a}{\sum}H_b\right)^{\perp}$, then

\begin{equation}
H_a=\underset{b\leq a}{\pig} S_b
\end{equation} 

\end{thm}

In Speed's review on the decomposition into interaction subspaces\cite{Speed} the proof of this result is done by induction and in Lauritzen's reference book on Graphical models (Appendix B \cite{Lauritzen}) the explicit expression of the projections on the interaction spaces is given, which is for any $a\subseteq I$,

\begin{equation}\label{chapter-3-interactions}
s_a=\underset{b\leq a}{\sum}(-1)^{|a\setminus b|}\pi_b
\end{equation}

We would like in this document to extend the interaction decomposition to Hilbert spaces. At first sight one could think that it is a particular case of an interaction decomposition compatible with a collection of projectors; one could argue that here Equation \ref{chapter-3-interactions} is simply the canonical decomposition with respect to the collection of projectors $(\pi_a,a\in \A)$. However it is not the case as we shall discuss in Section 5; the two points of view differ and as a consequence the results we shall present in Section 3 requiere much weaker constraints on the poset $\A$ than the one needed for collections of projectors.\\

There is an other example of such decomposition for a $\N$-filtered collection of subspaces of the space of all random variables which is the chaos decomposition. We will follow the presentation given in Sinai's \emph{Theory of Phase Transition: rigorous results} \cite{Sinai}. This construction makes it possible to characterize the space of potentials. Let us recall how one can define potentials for a non finite but countable $I$, and with $E_i=\R$ for $i\in I$.

\begin{nota}
For $m\in \N$, we note $[m]$ for $[1,...,m]$.
\end{nota}

Let $I=\Z^d$, $E=\R^I$; for any $\phi\in E$ and $x:[m]\to I$, let $\phi(x)=\underset{i\in [m]}{\prod}\phi(x_i)$. Let us note the set of maps of $I^{[m]}$ as $\A_m$.\\

Let us note $\Phi(x)$ the random variable that sends $\phi\in \R^I$ to $\phi(x)$.\\

\begin{nota}
By convention $[0]=\emptyset$ and $\underset{i\in \emptyset}{\prod}\phi(x_i)=1$; therefore $\A_0$ has one element.
\end{nota}

Let us assume that $P$ is a Gaussian distribution on $E$, for any $x\in \A$, $\Phi(x)\in L^2(E,P)$; for a collection $(v_j\in L^2(E,P) ,j\in J)$, we shall note $\langle v_j,j\in J\rangle$ the Hilbert subspace of $L^2(E,P)$ generated by these elements, i.e. the smallest Hilbert subspace of $ L^2(E,P)$ that contains them, its explicit expression is $\overline{\underset{j\in J}{\sum}\mathbb{K}v_j}$ when $L^2(E,P)$ is a $\mathbb{K}$-Hilbert space.\\

 For any $m\in \N$, let $H(m)=\langle\Psi(x),x\in \underset{k\leq m}{\cup}\A_m\rangle$ and $H=L^2(E,P)$.

\begin{lem} \label{c:density}

\begin{equation}
H= \overline{\underset{m\in \N}{\sum}H_m}
\end{equation}
\end{lem}

\begin{proof}
Theorem 2 \cite{T.Levy}

\end{proof}

\begin{prop}\label{context:simple-decomposition}
For $m\in \N$, let 
\begin{equation}
S_m=H(m)\cap H(m-1)^\perp
\end{equation}
 
then,
\begin{equation}
H(m)=\pig_{k\leq m}S_m
\end{equation} 

Furthermore $ L^2(E,P)=\pig_{m\in \N}S_m$

\end{prop}

\begin{proof}

Let us show the first statement by induction.\\

$H(\emptyset)=\mathbb{K}=S(\emptyset)$. Let us assume that for some $m\geq 0$, $H(m)=\pig_{0\leq k\leq m}S_k$; one has that $H(m+1)=S_{m+1}\bigoplus^{\perp} H(m)$ therefore,

\begin{equation}
H(m+1)=\bigoplus_{0\leq k \leq m+1}^\perp S_k
\end{equation}

which ends the proof by induction.\\

Therefore $\sum_{m\in \N}H_m= \bigoplus_{m\in \N}S_n$ and as $H= \overline{\sum_{m\in \N}H_m}$ (Lemma \ref{c:density}) therefore $H=\pig_{m\in \N}S_n$.

\end{proof}

\begin{defn}[Definition 4.7 \cite{Sinai}]\label{context:Hermite-Ito polynomials}

Let $P$ be a Gaussian distribution on $E$, let $m\in \N$, let $x\in \A_m$; $\Psi(x)$ has a unique orthogonal decomposition, in other words there is a unique couple $(u_1,u_2)$ such that $\Psi(x)=u_1+u_2$ with $u_1\in H(m-1)^\perp$, and $u_2\in H(m-1)$ with $\langle u_1,u_2\rangle=0$.

 We shall note $u_1$ as $:\Psi(x):$ and call it the Hermite-Ito polynomial of $\Psi(x)$.\\
\end{defn}

\begin{prop}
For any $m\in \N$ and $x\in \A_m$, then $S_m(\Psi(x))=:\Psi(x): $  , its Hermite-Ito polynomial. 

\end{prop}
\begin{proof}
By construction,

\begin{equation}
\Psi(x)= S_m(\Psi(x)) +\sum_{k <m} S_k(\Psi(x))
\end{equation} 

and therefore, $S_m(\Psi(x))=:\Psi(x): $

\end{proof}

Instead of introducing directly the main equivalence theorem of this document for the most general setting of functors in $\ih$ we decided, rather to start by restricting our attention to collections of Hilbert subspaces which will serve as a motivation and justification for the general setting.

\section{Necessary and sufficient condition for the interaction decomposition to hold for a collection of Hilbert subspaces}\label{chapitre-4-decomposition-collection-subspaces}

\subsection{Decomposability for collections of Hilbert subspaces}\label{d-h:Decomposability}

As stated in the last section, a common way to give an interaction decomposition for factor subspaces is to use the canonical scalar product on $\R^E$, where $E$ is a finite set. We shall state this result in a more general context by giving a necessary and sufficient condition for a collection of Hilbert subspaces over a well founded poset to be decomposable.

Let us start by defining what an interaction decomposition for a collection of Hilbert subspaces of a given Hilbert space is.\\

\begin{defn}[Decomposable collection of Hilbert subspaces]\label{d-h:decomposable}
Let $\A$ be any poset, let $H$ be a Hilbert space and let $(H_a,a\in \A)\in [\A,\vect H]$ be an increasing collection of subspaces of $H$. $(H_a,a\in\A)$ is said to be decomposable if there is a collection of subspaces of $H$, $(S_a\in \vect{H},a\in \A^{\Plus})$, such that for any $a\in \A$,

\begin{equation}
H=\underset{a\in \A^{\Plus}}{\pig} S_a
\end{equation}

and such that for any $a\in \A$,

\begin{equation}
H_a=\underset{b\leq a}{\pig} S_b
\end{equation}

We shall call $(S_a,a\in \A^{\Plus})$ a decomposition of $(H_a,a\in \A)$; by convention $H_1=H$.

\end{defn}

\subsection{Characterizing the decompositions}\label{d-h:Characterizing the decompositions}

\begin{prop}\label{d-h:interaction-subspaces}
Let $H$ be a Hilbert space and let $(H_a\in \vect{H},a\in \A)$ be a decomposable collection of Hilbert subspaces of $H$; let $(S_a,a\in \A^{\Plus})$ be a decomposition of $(H_a,a\in \A)$, then for any $a\in \A^{\Plus}$

\begin{equation}
S_a=H_a\cap \bigcap_{b\lneqq a}H_b^\perp
\end{equation}
\end{prop}

\begin{lem}\label{d-h:lemma-double-sum}
Let $I$ be any set and $(V_i,i\in I)$ a collection of vector subspaces, that are not necessarily closed, of a Hilbert space $H$, then 

\begin{equation}
\overline{\underset{i\in I}{\sum}V_i}=\overline{\underset{i\in I}{\sum}\overline{V_i}}
\end{equation}

\end{lem}
\begin{proof}
By definition $\overline{\underset{i\in I}{\sum}V_i}\subseteq \overline{\underset{i\in I}{\sum}\overline{V_i}}$; for any $i\in V_i$, $V_i\subseteq \overline{\underset{i\in I}{\sum}V_i}$ and therefore $\overline{\underset{i\in I}{\sum}\overline{V_i}}\subseteq \overline{\underset{i\in I}{\sum}V_i}$.
\end{proof}

\begin{rem}
The sum in Lemma \ref{d-h:lemma-double-sum} can be rewritten as,
\begin{equation}
\com \bigvee_{i\in I} V_i=\bigvee_{i\in I}\com V_i
\end{equation}

where $V_i$ are seen as pre-Hilbert spaces, and which is a consequence of $\com$ being left adjoint to the forgetful functor $U:\h \to \ph$ (Proposition \ref{a-h-c:adjoint}).
\end{rem}

Proof of Proposition \ref{d-h:interaction-subspaces}
\begin{proof}
let $(H_a\in \vect{H},a\in \A)$ be a decomposable collection of Hilbert subspaces of $H$.\\

Let us notice that $ \underset{b\lneq a}{\bigcap}H_b^\perp=\left(\underset{b\lneq a}{\sum}H_b\right)^\perp$. Furthermore, 
\begin{equation}
\overline{\underset{b\lneq a}{\sum}H_b}=\overline{\underset{b\lneq a}{\sum}\overline{\underset{c\leq b}{\sum} S_c}}
\end{equation}

By Lemma \ref{d-h:lemma-double-sum} $\overline{\underset{b\lneq a}{\sum}H_b}=\overline{\underset{b\lneq a}{\sum}\underset{c\leq b}{\sum} S_c}$ and one remarks that $\underset{b\lneq a}{\sum}\underset{c\leq b}{\sum} S_c= \underset{b\lneq a}{\sum}S_b$. Futhermore, as $H_a=\underset{b\leq a}{\pig}S_a$ one has that $S_a$ is orthogonal to $\underset{b\lneq a}{\sum}S_b$; therefore, $S_a=H_a\cap \left(\overline{\underset{b\lneq a}{\sum}H_b}\right)^\perp$ and so $S_a=H_a\cap \underset{b\lneq a}{\bigcap}H_b^\perp$.
\end{proof}

\begin{rem}
A consequence of Proposition \ref{d-h:interaction-subspaces} is that if $(H_a,a\in \A)$ is decomposable then it has a unique decomposition, i.e. if for any $a\in \A^{\Plus}$, $H_a=\underset{b\leq a}{\pig}S_a=\underset{b\leq a}{\pig}{S_1}_a$ then for any $a\in \A$, $S_a={S_1}_a$.
\end{rem}

\begin{nota}
Let $(H_a\in\vect{H},a\in \A)$ be a collection of Hilbert subspaces of a Hilbert space $H$, for $a\in\A^{\Plus}$ we shall note $s_a^\perp$ the orthogonal projection onto $S_a=H_a\cap \underset{b\lneqq a}{\bigcap}H_b^\perp$
\end{nota}

\subsection{Intersection property}\label{d-h:Intersection property}

\begin{defn}
Let $H$ be a Hilbert space, let $\A$ be any poset and let $(H_a\in \vect{H},a\in \A)$ be an increasing collection of subspace of $H$; for any subposet $\B$ of $\A$ we shall denote $H(\B)$ the completion of $\sum_{b\in \B}H_b$, i.e. $\overline{\sum_{b\in \B}H_b}$ (Appendix B Proposition \ref{a-h-c:completion-closure}). We shall denote $\pi(\B)$ the othogonal projection on $H(\B)$ and if $\B=\hat{a}$ with $a\in \A$ we shall simply note it as $\pi_a$.
\end{defn}

\begin{prop}\label{d-h:reorganise-sum}
Let $H$ be a Hilbert space, let $I$ be any set and let $(V_i,i\in I)$ be vector subspaces of $H$, i.e pre-Hilbert subspace of $H$. If $H=\pig_{i\in I}V_i$ then for any $J,J_1\subseteq I$ disjoint subsets of $I$, 

\begin{equation}
H=\pig_{j\in J} V_j\oplus^\perp \pig_{j\in J_1}V_j
\end{equation}

\end{prop}

\begin{proof}

Let $V=\pig_{j\in J}V_j$, $U=\pig_{j\in J_1}V_j$, one has that $H=V+U$ by Lemma \ref{d-h:lemma-double-sum} and bacause a finite sum of closed spaces is a closed.\\

Let $v\in V$ and $u\in U$ $\langle v,u\rangle=0$. Let $v\in \overline{V}$ and $u \in\overline{U}$, then there is $v_n\in V ,n\in \N$ and $u_n \in U,n\in\N$ such that $\lim_{n\to \infty}v_n= v$ and $\lim_{n\to \infty} u_n=u$. There is $N\in\N$ such that for any $n\geq N$,

\begin{equation}
\vert \langle u,v\rangle\vert=\vert \langle u,v\rangle-\langle u_n,v_n \rangle\vert\leq 2\Vert v-v_n\Vert \Vert u-u_n\Vert
\end{equation}

Therefore $\langle u,v\rangle=0$.
\end{proof}

\begin{defn}[Intersection property]\label{d-h:intersection-property}
Let $\A$ be any poset, let $H$ be a Hilbert space, let $(H_a\in \vect{H},a\in \A)$ be an increasing collection of subspace of $H$; $(H_a\in \vect{H},a\in \A)$ is said to verify the intersection property if,

\begin{equation}\tag{I}
\forall a,b\in \A\quad \pi(\hat{a}\cap \hat{b})=\pi_a\pi_b
\end{equation}
\end{defn}

\subsection{Theorem: equivalence between interaction decomposition and intersection property}\label{d-h:Theorem: equivalence between decomposition and intersection property}

\begin{lem}\label{d-h:necessary-condition}
Let $\A$ be any poset, let $H$ be a Hilbert space, let $(H_a\in \vect{H},a\in \A)$ be an increasing collection of subspace of $H$; if $(H_a,a\in \A)$ is decomposable, then $(H_a,a\in \A)$ satisfies the intersection property.
\end{lem}
\begin{proof}
Let $(H_a\in \vect{H},a\in \A)$ be a decomposable collection of subspace of $H$; let $(S_a,a\in \A)$ be its decomposition. Let $v\in \bigoplus_{c\leq a} S_c$ and $u\in S_b$,
\begin{equation}\label{decomposition:nedessary-condition-equation}
\langle v,\pi_a(u)\rangle=\langle v,u\rangle=1[b\leq a] \langle v,u\rangle
\end{equation}

Furthermore as $\sum_{c\leq a} S_c$ is dense in $H_a$, one has that for any $v\in H_a$, $\langle v,\pi_a(u)\rangle= 1[b\leq a] \langle v,u\rangle$; therefore $\pi_as_b^\perp= 1[b\leq a]s_b^\perp$.\\

Let $v\in \bigoplus_{c\not \in \hat{a}\cap\hat{b}}S_c$, and $u\in H$ then by Equation (\ref{decomposition:nedessary-condition-equation}) one has that,

\begin{equation}
\langle \pi_a\pi_bu,v \rangle= \langle u,\pi_b\pi_av \rangle=0
\end{equation} 

Therefore by Proposition \ref{d-h:reorganise-sum}, $\pi_a\pi_bu\in \pig_{c\in \hat{a}\cap\hat{b}}S_c=H(\hat{a}\cap\hat{b})$; furthermore for any $v\in H(\hat{a}\cap\hat{b})$ as $v\in H(a)\cap H(b)$,
\begin{equation}
\langle u,v\rangle=\langle\pi_b u,v\rangle=\langle\pi_bu,\pi_a v\rangle= \langle\pi_a\pi_b u,v\rangle
\end{equation}

Therefore, $\pi_a\pi_b=\pi(\hat{a}\cap\hat{b})$.
\end{proof}

\begin{lem}\label{d-h:orthogonal}
 Let $\A$ be any poset, let $H$ be a Hilbert space, let $(H_a\in \vect{H},a\in \A)$ be an increasing collection of subspace of $H$ that satisfies the intersection property, then for any $a,b\in \A^{+1}$,
 
 \begin{equation}
 s_a^\perp s_b^\perp=\delta_a(b)s_a^\perp
 \end{equation}
\end{lem}
\begin{proof}
If $a\neq b$ then $\hat{a}\cap \hat{b}\subsetneq \hat{b}$; for any $v,u\in H$,

\begin{equation}
\langle s_a^\perp (v),s_b^\perp (u)\rangle=\langle \pi_as_a^\perp (v),\pi_bs_b^\perp (u)\rangle=\langle \pi_b\pi_as_a^\perp (v),s_b^\perp (u)\rangle
\end{equation}

Therefore,
\begin{equation}
\langle s_a^\perp(v),s_b^\perp(u)\rangle=\langle\pi(\hat{a}\cap\hat{b})s_a^\perp (v),s_b^\perp (u)\rangle
\end{equation}

As $\pi(\hat{a}\cap\hat{b})s_a^\perp (v)\in \overline{\underset{c\lneq b}{\sum}H_c}$, by construction 
\begin{equation}
\langle\pi(\hat{a}\cap\hat{b})s_a^\perp (v),s_b^\perp (u)\rangle=0
\end{equation}

and so,
\begin{equation}
\langle s_a^\perp s_b(v),u\rangle=0
\end{equation}

As $s_a^\perp$ is a projector, ${s_a^\perp}^2=s_a^\perp$.
\end{proof}

\begin{thm}\label{d-h:thm}
Let $\A$ be a well-founded poset, let $H$ be a Hilbert space, let $(H_a\in \vect{H}, a\in \A)$ be an increasing collection of subspaces of $H$ indexed on $\A$; $(H_a\in \vect{H}, a\in \A)$ is decomposable if and only if it satisfies the intersection property.
\end{thm}

\begin{proof}
The necessary condition is Lemma \ref{d-h:necessary-condition}.\\

Let us show by transfinite induction on $\A$ that for any $a\in \A$, $H_a=\overline{\underset{b\leq a}{\sum}S_b}$.

Let $a\in \A$, let for any $b<a$, $H_b=\overline{\underset{c\leq b}{\sum}S_b}$, then by construction, $H_a=S_a\overset{\perp}{\bigoplus}\overline{\underset{b\lneqq a}{\sum}H_b}$. By Lemma \ref{d-h:lemma-double-sum}, $\overline{\underset{b\lneqq a}{\sum}H_b}=\overline{\underset{c\lneqq a}{\sum}S_c}$. This ends the proof by transfinite induction.\\

Lemma \ref{d-h:orthogonal} enable us to conclude that the sum is an othogonal one, which ends the proof.

\end{proof}

\begin{rem}
Once one has exhibited the good framework, i.e. the condition on the subspaces (Intersection property) and the condition on the poset (well founded), the proof of Theorem \ref{d-h:thm} is essentially repeating the one given in the particular case of factor spaces for the canonical scalar product as presented in Lemma 2.1 \cite{Speed}. However if one does not consider that the ambient space is a Hilbert space the methods for showing that the interaction decomposition exists if and only if the intersection property is satisfied are very different from the one present above as we will explain in Section \ref{d-h-2:No easy road}.
\end{rem}

\subsection{Simplification for meet semi-lattices}\label{d-h:Simplification for meet-lattices}

\begin{cor}\label{chapitre-4-decomposition-hilbert-cor-simple}
Let $\A$ be a well-founded meet semi-lattice. Let $H$ be a Hilbert space and $(H_a\in \vect{H},a\in \A)$ be a collection of subspace of $H$ such that for any $a,b\in \A$,

\begin{equation}
\pi_a\pi_b=\pi_{a\wedge b}
\end{equation}

Then $(H_a,a\in \A)$ is decomposable.
\end{cor}

\begin{proof}
When $\A$ is a semi-lattice $\widehat{a \cap b}=\hat{a}\cap \hat{b}$, therefore one concludes by Theorem \ref{d-h:thm}.
\end{proof}

\section{Extension to $\cat{IHilb}$}\label{chapitre-4-extension}

\subsection{Decomposability for functors in $\ih$}\label{d-h-2:Decomposability}
Let $H$ be a Hilbert space and $\F:\A\to \vect{H}$ be an increasing function; the poset $\vect{H}$ can be identified to the category that has as objects the Hilbert subspaces of $H$ and as morphisms the inclusion maps;  therefore $\F$ is a functor from $\A$ into this category. Furthermore the inclusions are in fact isometries, so $\vect{H}$ is a subcategory of $\ih$. We shall therefore consider in what follows, functors from a poset $\A$ to $\ih$.\\

Let us now try to find what is the good notion of decomposability for functors from a poset to $\ih$; in order to do so let us reformulate decomposability for functors from $\A$ to $\vect{H}$.\\

Let $\A$ be a poset, let $\cat{C}$ be any category; any object of $A$ of $\cat{C}$ induces a constant functor $iA:\A\to \cat{C}$, that we shall note as $A$.\\

For any collection of pre-Hilbert spaces, $(H_i,i\in I)$, one can define what we called its direct sum $\bigoplus_{i\in I}H_i$ (Appendix B Definition \ref{a-h-c:direct-sum}) and for any collection of pre-Hilbert spaces $(K_i,i\in I)$ and for any collection of morphism of $\ph$, $(\phi_i:H_i\to {H_1}_i,i\in I)$, we shall call $\phi$ the following morphism,

\begin{equation}
\begin{array}{ccccc}
\phi& : &\bigoplus_{i\in I}H_i & \to & \bigoplus_{i\in I}K_i\\
& & (v_i,i\in I) & \mapsto & (\phi_i(v_i),i\in I) \\
\end{array}
\end{equation}

$\phi$ is continuous and if $(\phi_i,i\in I)$ are isometries so is $\phi$.

\begin{prop}\label{d-h-2:decomposition-functor}
Let $I$ be any set and let $(G_i,i\in I)$ be a collection of functors from a poset $\A$ to $\cat{IPhilb}$; let for any $a\in \A$, 

\begin{equation}
(\underset{i\in I}{\bigoplus}G_i)(a)=\bigoplus_{i\in I}(G_i(a))
\end{equation}

and for any $a,b\in \A$ such that $b\leq a$,
\begin{equation}
(\underset{i\in I}{\bigoplus}G_i)(b\leq a)= (\underset{i\in I}{\bigoplus}G_i)^b_a=\underset{i\in I}{\bigoplus}{G_i}^b_a
\end{equation}

Then $\underset{i\in I}{\bigoplus}G_i$ is a functor from $\A$ to $\cat{IPhilb}$.
\end{prop}

\begin{proof}
Let $U_1:\ph\to \ve$ be the forgetful functor, let $a,b,c\in \A$ such that $c\leq b\leq a$ then $U_1(\bigoplus_{i\in I}{G_i}^b_a)=\bigoplus_{i\in I}{G_i}^b_a$, where the direct sum is taken in $\ve$; therefore $U_1 \bigoplus_{i\in I}{G_i}^b_a \circ U_1\bigoplus_{i\in I}{G_i}^c_b=U_1 \bigoplus_{i\in I}{G_i}^c_a$ and $\bigoplus_{i\in I}{G_i}^b_a \bigoplus_{i\in I}{G_i}^c_b=\bigoplus_{i\in I}{G_i}^c_a$.\\

For any $a,b\in \A$ such that $a\geq b$, $\bigoplus_{i\in I} {G_i}^b_a$ is an isometry as, for any $v\in \bigoplus_{i\in I}G_i(b)$, 

\begin{equation}
\langle \bigoplus_{i\in I} {G_i}^b_a(v),\bigoplus_{i\in I} {G_i}^b_a(v)\rangle = \sum_{i\in I} \langle {G_i}^b_a(v_i), {G_i}^b_a(v_i)\rangle
\end{equation}

\end{proof}

Let $G$ be a functor from a poset $\A$ to $\Vect$ or $\h$ or $\ph$, for $a\in \A$ we will denote $G1[a\leq .]$ the functor defined as for any $b\in \A$ such that $a\leq b$, 
$$G1[a\leq .](b)= G(b)$$

and if not $G1[a\leq .](b)=0$; for any $c\in \A$ such that $b\leq c$, if $a\leq b$ then,

$$G1[a\leq .]^b_c = G^b_c$$

\begin{prop}\label{d-h-2:decomposable}
Let $\A$ be any poset, let $H:\A\to \vect{H}$ be a functor; $(H_a,a\in \A)$ is decomposable in the sense of Definition \ref{d-h:decomposable} if and only if $H$ is isometrically isomorphic to $\com\underset{a\in \A}{\bigoplus} iS_a1[a\leq .]$, where $(S_a,a\in \A^{\Plus})$ is the decomposition of $(H_a,a\in \A)$.
\end{prop}

\begin{proof}
Let $(H_a,a\in \A)$ be decomposable, then for any $a\in \A$,

\begin{equation}
H_a=\com \bigoplus_{b\leq a} S_b= \com \bigoplus_{b\in \A} S_b(a) 1[b\leq a]
\end{equation}

In other words, the application that sends any $v\in H_a$ to $\bigoplus_{b\leq a} s_b^{\perp}(v)$ is an isomorphic isometry. Furthermore for $b\in \A$ such that $b\leq a$, the inclusion of $H(b)$ to $H(a)$ corresponds, up to this isomophic isometry, to $\com\underset{c\in \A}{\bigoplus} iS_c1[c\leq .]^b_a$.\\

The other way around, one only needs to remark that when $H_a$ is isometrically isomorphic to $\bigoplus_{b\leq a} S_b$ then, for any $b,c\in \A$ such that $b\leq a$ and $c\leq a$, and for $u\in S_b$, $v\in S_c$, then

\begin{equation}
\langle u, v\rangle = \delta_b(c) \langle u, v\rangle
\end{equation}
\end{proof}

As for any $a\in \A$, $S_a$ is a functor from $\A$ to $\ih$, Propositions \ref{d-h-2:decomposition-functor}, \ref{d-h-2:decomposable} unables us to conclude that decomposability can be directly stated for functor in $\ih$ as that it is  equal to a decomposition up to isomorphism; lets make this more formal.\\

\begin{defn}[Decomposability for functors in $\ih$]\label{d-h-3:def-decomposable}
Let $\A$ be any poset, a functor $G:\A\to\cat{IHilb}$ is decomposable if there is a collection of functors $(S_a:\A\to\core\ih,a\in \A)$ such that, 
\begin{equation}
G\cong \com \underset{a\in \A}{\bigoplus}S_a1[a\leq.]
\end{equation}

i.e. there is a natural transformation, $(\phi_a,a\in \A)$, where for any $a\in \A$, $\phi_a$ is an isometric isomorphism from $G(a)$ to $\com \bigoplus_{b\leq a}S_b(a)$. $(S_a 1[a\leq.],a\in \A)$ will be called a decomposition of $G$.
\end{defn}

\begin{prop}\label{d-h-2:thm-lem}
Let $\A$ be any poset and $G$ be a functor from $\A$ to $\ih$. If $G$ is decomposable then there is $H$ a Hilbert space and a functor $\U:\A\to \vect H$ such that $G\cong \U$.

\end{prop}
\begin{proof}

Let $G$ be decomposable and let $(S_a,a\in \A)$ be a decomposition of $c\in \A$, then $S_c(c)1[c\leq .]\cong S_c$ therefore $\com\bigoplus_{c\in \A} S_c \cong \bigoplus_{c\in \A} S_c(c) 1[c\leq .]$ and for any $a\in\A$, $\cong \bigoplus_{c\leq a} S_c(c)$ is a subspace of $\bigoplus_{c\in \A} S_c(a)$ which ends the proof.

\end{proof}

\subsection{No short-cut}\label{d-h-2:No easy road}

If any functor from $\A$ to $\ih$ was isometrically isomorphic to a functor from $\A\to \vect{H}$ for some Hilbert space $H$, the last section would give an answer to when one can say that a functor in $\ih$ is decomposable; however we will now explain why it is not the case.\\

\begin{nota}
Let us note $\f$ the forgetful functor from $\ih$ to $\cat{Vect}$.
\end{nota}

\begin{prop}
Let $\A$ be any poset and $G:\A\to \cat{IHilb}$ be a functor, then $\eta:VG\to\colim_{a\in \A}\f G$ is a monomorphism of functors of $\ve$.
\end{prop}

\begin{proof}
For any $a,b\in \A$ such that $b\leq a$, ${G^b_a}^\dagger G^b_a=\id$; therefore $G^b_a$ is injective and so is $\f G^b_a$, therefore by Proposition 2.1 \cite{GS2}, $\eta_a:G(a)\to\underset{a\in \A}{\colim}\f G$ is injective.\\
\end{proof}

For any pre-Hilbert space $H$ and any closed pre-Hilbert subspace $H_1$ of $H$, one can define the pre-Hilbert space $H/H_1$ (see Appendix Proposition \ref{a-h-c:quotient-map}).

\begin{nota}
Let $G$ be a functor from a category $\cat{C}$ to $\cat{Vect}$; we shall note $\ker G$ the vector subspace of $\underset{a\in \A}{\bigoplus}G(a)$ generated by the collection $(G^b_a(v_b)\times a-v_b\times b \vert (b\leq a, v_b\in G(b)))$. 
\end{nota}

For a functor $G:\A\to \ih$, $\eta :G\to \colim_{a\in \A}VG$ induces a morphism $\rho:G\to \bigoplus G(a)/ \overline{\ker G}$, where $\overline{\ker G}\subseteq \bigoplus G(a)$. Let us remark that $\ker G$ is closed in $\bigoplus G(a)$, but we want to insist on this property so we will refer to it as $\overline{\ker G}$. By composition with the injection from $\bigoplus G(a)/ \overline{\ker G}$ to $\com (\bigoplus G(a)/ \overline{\ker G})$, it induces a morphism $\psi:G\to \com (\bigoplus G(a)/ \overline{\ker G})$. Let us remark that by Proposition \ref{a-h-c:adjoint-colimit} and the fact that $\com \ker G =\com  \overline{\ker G}$, one has that,

\begin{equation}
\com (\bigoplus G(a)/ \overline{\ker G})=\com \bigoplus G(a)/\com \ker G
\end{equation}

 In general $\psi$ is not monomorphism and even when it is, it is not a collection of isometries. \\

Indeed in general $\ker G$ is not closed in $\com \bigoplus_{a\in \A}G(a)$ and we shall now give a counter example. Let $\A=\N$ and $H$ a Hilbert space, for any $n\in \N$ let $G(n)=H$ and for any $m\geq n$ let $G^n_m=\id$. The orthogonal to $\ker G$ in $\com \bigoplus_{n\in \N}H$ is $0$ as we shall show. Any collection $(v_n,n\in \N)$ is in $\ker G^{\perp}$ if and only if, for any $u\in H$ and $n\in \N$,

\begin{equation}
\langle v_n,u\rangle=\langle v_{n+1},u\rangle
\end{equation}

Therefore $(v_n,n\in \N)$ is a constant sequence; however $\lim_{n\to\infty}||v_n||=0$ therefore $(v_n,n\in \N)=0$. Then $\overline{\ker G}=\com \bigoplus_{a\in \A}G(a)$ and $\ker G$ is dense in $\bigoplus_{a\in \A}G(a)$; finaly let $v \in H$; pose $v_0=v$ and $v_n=0$ for $n\in \N$, $v$ is not in $\ker G$. In this example, $\psi_a=0$ for any $a\in \A$.\\

Let us now show that even when $\psi$ is injective, it is in general not a morphism of functor from $\A$ to $\ih$.\\

\begin{prop}\label{d-h-2:not-isometry}
Let $\A$ be any poset, let us consider a funtor $G:\A\to \cat{IHilb}$. $\com G/\com \ker G$ is isometrically isomorphic to $(\com \ker G)^\perp\subseteq \com \bigoplus_{a\in \A}G(a)$.
\end{prop}

\begin{proof}
Let $p: \com \bigoplus_{a\in \A}G(a) \to (\com \ker G)^\perp$ be the orhogonal projection on $(\com \ker G)^\perp$; by definition it is surjective furthermore for any $v\in \com \bigoplus_{a\in \A} G(a)$, 
\begin{equation}
\Vert [v]\Vert=\underset{u\in \ker G}{\inf} \Vert v+u\Vert=\underset{u\in \overline{\ker G}}{\inf} \Vert v+u\Vert
\end{equation}

and, 

\begin{equation}
\underset{u\in \overline{\ker G}}{\inf} \Vert v+u\Vert=\Vert p(v)\Vert
\end{equation}

which ends the proof.

\end{proof}

Let $G$ be a functor from $\A$ to $\ih$, let $p: \com \bigoplus_{a\in \A}G(a) \to {\com \ker G}^\perp$ be the orhogonal projection on $(\com \ker G)^\perp\subseteq \com \bigoplus_{a\in \A} G(a)$. Let $a,b\in \A$ with $b< a$, let $v\in G(b)$,then

\begin{equation}
\langle v, v\times b - G^b_a(v)\times a\rangle=\Vert v\Vert^2 
\end{equation}

Therefore $v$ is not in $(\com \ker G)^\perp$ unless $||v||=0$ and if $||v||\neq 0$, by Proposition \ref{d-h-2:not-isometry},

\begin{equation}\label{d-h-2:contraction-injection}
\Vert \psi_b(v)\Vert =\Vert p(v)\Vert < \Vert v\Vert
\end{equation}

Therefore if $\A$ has at least two elements, $a,b\in \A$, such that $b<a$, then the morphisms of $\psi$ are not morphisms of $\ih$, i.e. are not isometries; in fact the only case where $\psi$ is an isometry is when the poset is a disjoint union of elements $I$, i.e. for $x,y\in I$,

\begin{equation}
x\leq y  \text{ if and only if } x=y
\end{equation}
 and in this case the colimit is a direct sum of Hilbert spaces.\\

 This is the reason why, contrarily to the decomposition in $\ve$ as discussed in \cite{GS2}, we can't directly apply Theorem \ref{d-h:thm} to extend this theorem to $\ih$. We shall therefore follow the (longer) path inspired by \cite{GS3} to show the an extension still holds.

\subsection{Extension to $\A_1$}\label{d-h-2:Main construction 1}
 Let us first give the idea of the construction we will present in what follows. We know the Theorem \ref{d-h:thm} holds for subpaces of a given Hilbert space, therefore for $H$ a functor from a poset $\A$ to $\ih$, we will want to nest the Hilbert spaces $(H_a,a\leq \alpha)$ inside a given Hilbert space $H_\alpha$ when this Theorem holds. What we will show is that nesting is functorial in $\ih$ and that the decomposition of each $H_\alpha$ are connected to each other in such way that we can rebuild a  decomposition (in the sense of Definition \ref{d-h-3:def-decomposable}) for $H$.\\

Let $\A_1$ be the subposet of $\A\times \A$ which elements are couples $(\alpha, a)$ such that $a\leq \alpha$.

\begin{rem}
Let us recall that if $\phi$ is an isometry then $\im \phi$, in the sense of its image in $\ve$, is a Hilbert space.
\end{rem}

For any poset $\A$, any $\alpha \in \A$ and for $G$ a functor from $\A$ to $\ih$, let us call $G^{\hat{\alpha}}_\alpha: G|_{\hat{\alpha}}\to G(a)$ the monomorphism induced by $G$. In \cite{GS3} we introduced the category $\Split$ that has as objects vector spaces and as morphisms between two objects $V, V_1$ couples of linear applications $s:M\to M_1$, $r:M_1\to M$ such that,

$$ rs =\id$$

In \cite{GS3} we argued that a functor $H$ from a poset $\A$ to $\Split$ is uniquely identified by a couple of functor/ presheaf, $(G,F)$, such that for any $a,b\in \A$,

\begin{equation}
F^a_bG^b_a=\id    
\end{equation}

where $F^a_b$ denotes $F(b\leq a)$ and $G^b_a$ denotes $G(b\leq a)$. In our case $(G,G^{\dagger})$ is a functor from $\A$ to $\Split$. 
For a function $f: A\to B$ between two sets $A$,$B$ such that for $C\subseteq A$ and $D\subseteq B$ and $f(C)\subseteq D$, we will denote $f|^D_C: C\to D$ its restriction to $C$ and $D$. 

For $\alpha\geq \beta \geq a\geq b$ elements of $\A$, $\im G^b_\alpha\subseteq \im G^a_\alpha$; we shall note the inclusion as $L^{\alpha b}_{\alpha a}$. Let $L^{\beta a}_{\alpha a}=G^{\beta}_{\alpha}|^{\im G^a_\alpha}_{\im G^a_\beta}$. We will call $L$ the left coupling of $G$ (for motivation of this denomination see Definition 4.4 \cite{GS3}).

\begin{prop}\label{d-h-2:G-isomorphism}
Let $G$ be a functor from $\A$ to $\cat{IHilb}$, let $\alpha \in \A$, then $G^{\hat{\alpha}}_\alpha|^{L(\alpha,.)}$ is an isomorphism.
\end{prop}
\begin{proof}
Let us note $H=G^{\hat{\alpha}}_\alpha|^{L(\alpha,.)}$ for any $a\leq \alpha$, $H(a)$ is surjective and an isometry, therefore it is an isomorphism

\end{proof}

\begin{prop}\label{extension-left}
Let $G$ be a functor from $\A$ to $\cat{IHilb}$, let $L$ be its left coupling; $L$ is a functor from $\A_1$ to $\cat{IHilb}$.

\end{prop}

\begin{proof}

Let $\A$ be any poset and $G:\A\to \ih$ be a functor from $\A$ to $\ih$ ; let $(\alpha,a),(\beta,a)\in \A_1$ be such that $\alpha\geq \beta$, the following diagram commutes

\begin{equation}
\begin{tikzpicture}[baseline=(current  bounding  box.center),node distance=2cm, auto]
\node (A) {$G(a)$ };
\node (B) [below of=A] {$G(\beta)$};
\node (C) [right of =B] {$G(\alpha)$};
\node (D) [right of =A] {$G(a)$};
\draw[->] (A) to node [left] {$G^a_\beta $} (B);
\draw[->] (B) to node {$G^\beta_\alpha$} (C);
\draw[->] (A) to node {$\id $} (D);
\draw[->] (D) to node {$G^a_\alpha $} (C);
\end{tikzpicture}
\end{equation}

and by Proposition \ref{d-h-2:G-isomorphism}, $G^{a}_{\alpha}|^{L(\alpha,a)}$ and $G^a_\beta |^{L(\beta,a)}$ are isomorphims, therefore $L^{\beta a}_{\alpha a}=G^{\beta}_{\alpha}|^{L(\beta,a)}_{L(\alpha,a)}$ is an isomorphism of $\ih$.  
\end{proof}

\subsection{Extension to $\A_2$}\label{d-h-c:Main contruction 2}
\begin{defn}[Complete join-lattice]\label{d-h-2:complete-join-lattice}
A complete join-lattice, $\A$, is a poset that has a join for any subset of the poset, i.e. for any set $I$ and collection $(a_i\in \A,i\in I)$ there is $\vee_{i\in I}a_i$ such that for all $c\in \A$ such that,
\begin{equation}
\forall i\in I \quad c\geq a_i 
\end{equation}

then one has that $\vee_{i\in I}a_i \leq c$.
\end{defn}

\begin{rem}
Let $G$ be a functor from $\A$ to $\ih$, for any $a,b\in \A$ such that $a\geq b$, as $G^b_a$ is injective, it is a suboject of $G(a)$ (Appendix B Definition \ref{a-h-c:subobject}, Proposition \ref{a-h-c:mono-epi}) and a subspace of $G(a)$ up to isomorphism. 
\end{rem}

\begin{defn}

Let $\A$ be any poset and let $G$ be a functor from $\A$ to $\ih$; we shall note for any $\B\in \U(\alpha)$,
\begin{equation}
G^{\alpha}(\B)=\underset{a\in \B}{\bigvee}G^a_\alpha
\end{equation}

and for any $\B_1\in \U(\alpha)$ such that $\B_1\subseteq \B$, we shall note the inclusion $G^{\alpha}(\B_1)\leq G^{\alpha}(\B)$ as $\hG^{\alpha \B_1}_{\alpha \B}$. We shall also denote $G^{\alpha}(\B)$ as $\hG(\alpha, \B)$.\\

By Annex B Corollary \ref{a-h-c:join-lattice-isometry}, a representant of $G^\alpha (\B)$ is $\overline{\sum_{b\in \B}L(\alpha,b)}$ which reduces to $L(\alpha,b)$ when $\B=\hat{b}$. From now on we will consider that $G^\alpha (\B)$ is $\overline{\sum_{b\in \B}L(\alpha,b)}$.\\
 
 For $\alpha,\beta\in \A$ such that $\alpha\geq \beta$ and for $\B\in \U(\beta)$, let,
 
 \begin{equation}
 \hG^{\beta \B}_{\alpha \B}=G^{\beta}_{\alpha}|^{\hG(\beta, \B)}_{\hG(\alpha,\B)}
 \end{equation}
 
\end{defn}

\begin{rem}
Let us recall that colimits do not exist in general in $\ih$ (Appendix B Proposition \ref{a-h-c:no-colimit}); if $\A$ is a poset, $a\in \A$, $\B\in \U(\A)$ and $G:\A\to \ih$ a functor from $\A$ to $\ih$ then $\bigvee_{b\in \B}G^b_{a}$ is in general not a colimit of $G|_{\hat{a}\setminus a}$.
 \end{rem}

\begin{prop}\label{d-h-2:extension-2}
For any poset $\A$ and functor $G:\A\to \ih$, $\hG$ extends to a unique functor from $\A_2\to \ih$. 
\end{prop}

\begin{lem}\label{extension-functor-bigger-poset}
Let $\A$ be a poset, $\cat{C}$ be any category; let $M_1=\{\left((\alpha,\B),(\alpha,\B_1)\right): (\alpha,\B),(\alpha,\B_1)\in \A_2 \text{ and } \B\geq \B_1\}$, $M_2=\{\left((\alpha,\B),(\beta,\B)\right): (\alpha,\B),(\beta,\B)\in \A_2 \text{ and } \alpha\geq \beta\}$, and let $(G^{i}_{j}; i,j\in M_1\cup M_2: i\leq j)$ be such that for any $(\alpha,\B),(\alpha,\B_1),(\alpha,\B_2)\in \A_1$ such that $(\alpha,\B)\geq (\alpha,\B_1)\geq (\alpha,\B_2)$,
 
\begin{equation} 
G^{\alpha \B_1}_{\alpha \B}G^{\alpha \B_2}_{\alpha \B_1}= G^{\alpha \B_2}_{\alpha \B_1}
\end{equation}
 
 for any $(\alpha,\B),(\beta,\B),(\gamma,\B)$ such that $(\alpha,\B)\geq (\beta,\B)\geq (\gamma,\B)$,
 
\begin{equation}
 G^{\beta \B}_{\alpha \B} G^{\gamma \B}_{\beta \B}= G^{\gamma \B}_{\alpha \B}
\end{equation}

and for any $(\alpha, \B),(\alpha, \B_1),(\beta, \B),(\beta, \B_1)$ such that $(\alpha, \B)\geq (\alpha, \B_1)\geq (\beta, \B_1)$ and $(\alpha, \B)\geq (\beta, \B)\geq (\beta, \B_1)$, i.e $\hat{\alpha}\geq \hat{\beta} \geq \B\geq \B_1$,

\begin{equation}
G^{\beta \B}_{\alpha \B} G^{\beta \B_1}_{\beta \B}= G^{\alpha \B_1}_{\alpha \B}G^{\beta \B_1}_{\alpha \B_1}
\end{equation}

Then $G$ extends into a unique functor $G_1:\A_2\to \cat{C}$, which shall also be denote as $G$.

\end{lem}

\begin{proof}
The proof is essentially rewriting the proof of Proposition 4.8 \cite{GS3} replacing $c$ by $\B_2$, $b$ by $\B_1$ and $a$ by $\B$.
\end{proof}

Proof of Proposition \ref{d-h-2:extension-2}.
\begin{proof}

By Lemma \ref{extension-functor-bigger-poset} $\hG$ is a functor and for any $(\beta, \B_1)$, $(\alpha, \B)\in \A_2$ such that $(\beta, \B_1)\leq (\alpha, \B)$, $\hG^{\beta \B_1}_{\beta \B}$ is an inclusion therefore an isometry and $\hG^{\beta \B}_{\alpha \B}$ is a restriction of $G^\beta_\alpha$, which is an isometry, on its domain and codomain (it is even an isomorphism); therefore $\hG^{\beta \B_1}_{\alpha \B}$.

\end{proof}

\begin{rem}
If one notes $j:\A\to \A_1$ the increasing function such that $j(a)=(a,a)$ and $k:\A\to \A_2$ the increasing function such that $k(a)=(a,\hat{a})$ then 
\begin{equation}
Lk=\hG j\cong G
\end{equation} 
\end{rem}

\subsection{Intersection property}\label{Intersection property}

\begin{rem}
Let $\A$ be any poset and $G$ be a functor from $\A$ to $\ih$. Let $R$ be defined as for any $(\alpha, a),(\beta,b)\in \A_1$ such that $(\alpha,a)\geq (\beta,b)$,

\begin{equation}
R^{\alpha a}_{\beta b}= {L^{\beta b}_{\alpha a} }^{\dagger}
\end{equation}

Then $F$ is a presheaf.\\

Similarly let $\hF$ be defined as for any $(\alpha,\B),(\beta,\B_1)\in \A_2$ such that $(\alpha,\B)\geq (\beta, \B_1)$, 

\begin{equation}
\hF^{\alpha\B}_{\beta \B_1}=\hG^{\alpha \B}_{\beta \B_1}{}^\dagger
\end{equation}
Then $\hF$ is a presheaf.

\end{rem}

\begin{defn}[Intersection property]
Let $\A$ be any poset and $G$ be a functor from $\A$ to $\ih$. For any $(\alpha,\B) \in \A_2$ let us note $\pi^{\alpha}(\B)$ the orthogonal projection of $G(\alpha)$ onto $\hat{G}(\alpha, \B)$.\\

We shall say that $G$ satisfies the intersection property if for any 

\begin{equation}\tag{I'}
\forall (\alpha,a),(\alpha,b)\in \A_1,\quad \pi^{\alpha}(\hat{a}\cap \hat{b})=\pi^{\alpha}(\hat{a})\pi^{\alpha}(\hat{b})
\end{equation}

\end{defn}

\begin{prop}
Let $G$ be a functor from a poset $\A$ to $\ih$, $G$ satisfies the intersection property if and only if

\begin{equation}
\forall \alpha,a,b\in \A \text{ s.t. } a\leq \alpha,b\leq \alpha,\quad   G^a_\alpha {G^a_\alpha}^{\dagger}G^b_\alpha{G^b_\alpha}^{\dagger}= \hG^{\alpha \hat{a}\cap\hat{b}}_{\alpha \hat{\alpha}} \hG^{\alpha \hat{a}\cap\hat{b}}_{\alpha \hat{\alpha}}{}^{\dagger}
\end{equation}

\end{prop}

\begin{proof}

For any $(\alpha, b)\in \A_1$, $\phi_b:G(b)\to G(\alpha,b)$ is an isomorphism of $\ih$ and $G^{b}_{\alpha}=\phi_{\alpha}G^{\alpha b}_{\alpha \alpha} \phi_b^{-1}$; as $G^{\alpha b}_{\alpha \alpha}$ is an inclusion, the orthogonal projection $\pi^{\alpha}(b)= {G^{\alpha b}_{\alpha \alpha}}^{\dagger}$. One then has that,

\begin{equation}
G^b_\alpha{G^b_\alpha}^{\dagger}=\phi_{\alpha} G^b_\alpha  \phi_b^{-1}\phi_b^\dagger {G^b_\alpha}^\dagger{\phi_\alpha^{-1}}^\dagger
\end{equation}

and as $\phi_b$ is an isometry and $\phi_{\alpha}=id$,

\begin{equation}
G^b_\alpha{G^b_\alpha}^{\dagger}=\ G^b_\alpha  {G^b_\alpha}^\dagger
\end{equation}

which ends the proof.

\end{proof}

\begin{prop}
Let $\A$ be any poset and $G:\A\to \ih$ be a functor from $\A$ to $\ih$; let us assume that for some $\alpha \in \A$, and for any $a,b\in \A$ such that $a\leq \alpha $ and $b\leq \alpha$,
\begin{equation}
\pi^{\alpha}(\hat{a}\cap\hat{b})=\pi^\alpha(\hat{a})\pi^\alpha(\hat{b})
\end{equation}
then for any $\alpha_1\leq \alpha$, and any $a,b\leq \alpha_1$ one has that,

\begin{equation}
\pi^{\alpha_1}(\hat{a}\cap\hat{b})=\pi^{\alpha_1}(\hat{a})\pi^{\alpha_1}(\hat{b})
\end{equation}
\end{prop}


\subsection{Predecomposition and natural transformations}\label{d-h-2:Decomposition natural transformations}

\begin{defn}
Let $\A$ be a poset and $G$ be a functor from $\A$ to $\ih$; for any $(\alpha,a) \in \A_1$, when $c\leq a$, let,

\begin{equation}
S_c(\alpha,a)=L(\alpha,c)\cap \bigcap_{d\lneqq c}L(\alpha,d)^\perp 
\end{equation}

otherwise

\begin{equation}
S_c(\alpha,a)=0
\end{equation}

Futhermore let $(\alpha, a),(\beta,b)\in \A_1$ be such that $(\alpha,a)\geq (\beta,b)$, for any $v\in S_c(\beta,b)$, let

\begin{equation}
{S_c}^{\beta b}_{\alpha a}(v)= G^{\beta}_{\alpha}(v)
\end{equation}

This reduces to  ${S_c}^{\beta b}_{\alpha a}= 0$ when $c\not\leq b$.\\

Let for any $(\alpha,a)\in \A_1$, $s_c(\alpha,a):L(\alpha,a)\to S_c(\alpha,a)$ be the orhogonal projection on $S_c(\alpha,a)$. Furthermore we shall denote $V_c$ the restriction of $S_c$ to $\A$, i.e. for any $a\in \A$,

\begin{equation}
V_c(a)= S_c(a,a)
\end{equation}

and for $b\in \A$ such that $b\leq a$,

\begin{equation}
{V_c}^b_a= {S_c}^{bb}_{aa}
\end{equation}

\end{defn}
\begin{proof}
Let $\alpha,\beta,a\in \A$ such that $\alpha \geq \beta \geq a$, let $b\in \A$ such that $b\lneqq a$, let $v\in S_b(\beta,a)$, let $w\in L(\alpha,b) $ then as $L^{\beta b}_{\alpha b}$ is an isomorphism of $\ih$ (Proposition \ref{d-h-2:G-isomorphism}) one has that there is $w_1\in L(\beta,b)$ such that $w=G^\beta_\alpha(v)$ and,

\begin{equation}
\langle G^\beta_\alpha(v), w\rangle = \langle G^\beta_\alpha(v),G^\beta_\alpha(w_1)\rangle= \langle s_a(v), w_1\rangle=0
\end{equation}

This shows that ${S_c}^{ \beta a}_{\alpha a}$ is well defined and therefore the $S_c$ is well defined.

\end{proof}

The collection $(V_a,a\in \A)$ is always well defined even when the functor is not decomposable, we shall call it the predecomposition of $G$ as it is analogous to predecompositions for functors from $\A$ to $\Vect$ introduced in \cite{GS2}; they are the bits from which one can build a decomposition of $G$ when $G$ is decomposable.

\begin{prop}\label{decomposition-hilb:projection-commutation}
Let $\A$ be any poset and $G:\A\to \ih$ be a functor, then for any $c\in \A$, $S_c$ is a functor from $\A_1$ to $\ih$. Furthermore for any $c\in \A$, $s_c:L\to S_c$ is a natural transformation of funtors of $\h$, i.e. for any $(\alpha,a),(\beta,b)\in \A_1$ such that $(\alpha,a)\geq (\beta,b)$,

\begin{equation}
{S_c}^{\beta b}_{\alpha a} s_c(\beta, b)= s_c(\alpha, a) L^{\beta b}_{\alpha a}
\end{equation}
\end{prop}

\begin{proof}

Let $(\alpha, a),(\beta,b)\in \A$ be such that $(\alpha,a)\geq (\beta,b)$, as $G^{\beta}_{\alpha}$ is an isometry so is ${S_c}^{\beta b}_{\alpha a}$; for $(\gamma, g)\in \A_1$ such that $(\gamma,g)\leq (\beta,b)$ and for $v\in S_d(\gamma,g)$ one has that,

\begin{equation}
{S_c}^{\beta b}_{\alpha a}{S_c}^{\gamma g}_{\beta b}(v)= G^\beta_\alpha G^\gamma_\beta(v)= G^\gamma_\alpha(v)
\end{equation}

\begin{equation}
{S_c}^{\beta b}_{\alpha a}{S_c}^{\gamma g}_{\beta b}(v)= {S_c}^{\gamma g}_{\alpha a}
\end{equation}

Let $(\alpha,a),(\beta, a)\in \A_1$ be such that $\alpha\geq \beta$, let $v\in L(\beta,a)$ and $w\in S_c(\beta,a)$; let us recall that ${L^{\beta a}_{\alpha}}^\dagger$ is an isomorphism of $\ih$ (Proposition \ref{d-h-2:G-isomorphism}), therefore,

\begin{equation}
\langle {L^{\beta a}_{\alpha a}}^\dagger s_c(\alpha, a) L^{\beta a}_{\alpha a}(v),w\rangle =\langle  s_c(\alpha, a) L^{\beta a}_{\alpha a}(v), L^{\beta a}_{\alpha a}(w)\rangle 
\end{equation}

As $s_c(\alpha, a)$ is the orthogonal projection on $S_c(\alpha,a)\subseteq L(\alpha,a)$ one has that,

\begin{equation}
\langle {L^{\beta a}_{\alpha a}}^\dagger s_c(\alpha, a) L^{\beta a}_{\alpha a}(v),w\rangle =\langle L^{\beta a}_{\alpha a}(v), L^{\beta a}_{\alpha a}(w)\rangle=\langle v, w\rangle
\end{equation}

Therefore by definition of $s_c(\beta, a)$,

\begin{equation}
s_c(\beta, a)(v)= {L^{\beta a}_{\alpha a}}^\dagger s_c(\alpha, a) L^{\beta a}_{\alpha a}(v)
\end{equation}

and, 

\begin{equation}
{S_c}^{\beta a}_{\alpha a} s_c(\beta, a)=s_c(\alpha, a) L^{\beta a}_{\alpha a}
\end{equation}

Furthermore by definition one has that for $(\alpha,a),(\alpha,b)\in \A_1$ such that $a\geq b$,

\begin{equation}
{S_c}^{\alpha b}_{\alpha a} s_c(\alpha, b)=s_c(\alpha, a) L^{\alpha b}_{\alpha a}
\end{equation}

and so for any $(\alpha,a),(\beta,b)\in \A_1$ such that $(\alpha,a)\geq (\beta,b)$,

\begin{equation}
{S_c}^{\beta b}_{\alpha a} s_c(\beta, b)= {S_c}^{\beta a}_{\alpha a}{S_c}^{\beta b}_{\beta a}s_c(\beta,b)=s_c(\alpha, a)L^{\beta a}_{\alpha a} L^{\beta b}_{\beta a}
\end{equation}

\begin{equation}
{S_c}^{\beta b}_{\alpha a} s_c(\beta, b)= s_c(\alpha, a) L^{\beta b}_{\alpha a}
\end{equation}

\end{proof}

\subsection{Main Theorem}\label{d-h-2:Main Theorem}

\begin{thm}\label{d-h-2:thm}
Let $\A$ be a well founded poset and $G$ be a functor from $\A$ to $\ih$, $G$ is decomposable if and only if it statifies the intersection property, i.e.

\begin{equation}
\forall \alpha,a,b\in \A \text{ s.t. } a\leq \alpha,b\leq \alpha,\quad   G^a_\alpha {G^a_\alpha}^{\dagger}G^b_\alpha{G^b_\alpha}^{\dagger}= \hG^{\alpha \hat{a}\cap\hat{b}}_{\alpha \hat{\alpha}} \hG^{\alpha \hat{a}\cap\hat{b}}_{\alpha \hat{\alpha}}{}^{\dagger}
\end{equation}

\end{thm}

\begin{lem}\label{decomp-hilb:well-founded-sub-poset}
Let $\A$ be a well-founded poset then any subposet of $\A$ is well founded.

\end{lem}

\begin{proof}
Proposition 5.3 \cite{GS2}.

\end{proof}

\begin{proof}

Let $\alpha\in \A$, by Lemma \ref{decomp-hilb:well-founded-sub-poset} and by Theorem \ref{d-h:thm} one has that $\bigoplus_{d\in \A}s_d(\alpha,.): G(\alpha,.)\to \bigoplus_{d\in \A}S_d(\alpha,.)$ is an isomorphims and an isometry. By definition for any $a,b,c\in \A$ such that $\alpha\geq a\geq b\geq c$, ${S_c}^{\alpha b}_{\alpha a}$ is the identity application. Furthermore for any $a\in \A$ such that $a\leq \alpha$, $\phi_a:G(a)\to G(\alpha, a)$ is an isomorphism and therefore by Lemma \ref{decomposition-hilb:projection-commutation} $\bigoplus_{d\in \A}{S_d}^{\alpha a}_{a a}$ is an isomorphism. Then for any $d\in \A$, ${S_d}^{\alpha a}_{aa}$ is an isomorphism and so as,

\begin{equation}
{S_d}^{\alpha \alpha}_{\alpha a}{S_d}^{\alpha a}_{a a}= {S_d}^{\alpha \alpha}_{a a}= {S_d}^\alpha_a
\end{equation}

${S_d}^\alpha_a$ is an isomophism which shows that the intersection property implies decomposable.

The necessary condition is a consequence of Proposition \ref{d-h-2:thm-lem} and Theorem \ref{d-h:thm}.

\end{proof}

\section{How to relate the interaction decomposition for Hilbert spaces over a finite poset to the interaction decomposition for presheaves?}\label{bootstrap}

In this section we will limit our attention to finite posets $\A$, and to subspaces, $(H_a,a\in \A)$, of a given Hilbert space $H$; we want to give a brief overview of how and when to build a canonical decomposition for the collection of orthogonal projectors $(\pi_a,a\in \A)$. It is a different point of view from the one presented in Section 2 and we shall show that both of the points of view are equivalent when $(H_a,a\in \A)$ is decomposable or $(\pi_a,a\in \A)$ is decomposable, but that if not hey aren't equivalent.

\subsection{How to relate both interaction decompositions}

\begin{prop}\label{c-d:decomposable-implies-projector}
Let $H$ be a Hilbert spaces, and $\A$ a finite poset; if $(H_a\in \vect{H})$ is decomposable then $(\pi_a,a\in \A)$ is decomposable and for any $a\in \A$, 

\begin{equation}
s_a=s_a^\perp
\end{equation}
\end{prop}

\begin{proof}
Let $(H_a,a\in \A)$ be decomposable therefore by definition, for any $a\in \A$,

\begin{equation}
\pi_a=\sum_{b\leq a}s^\perp_a
\end{equation}

One also has by definition that,

\begin{equation}
\pi_a=\sum_{b\leq a}s_a
\end{equation}
As $\mu_\A$ is injective, $s_a=s_a^\perp$ and therefore $(\pi_a,a\in \A)$ is also decomposable.  
\end{proof}

\begin{prop}\label{c-d:projector-decomposable-implies}
Let $\A$ be a finite poset, let $H$ be a Hilbert space, let $(H_a,a\in \A)$ be a collection of Hilbert subspace of $H$. Suppose that $(\pi_a,a\in \A)$ is decomposable, where $\pi_a$ is the orthogonal projection on $H_a$, then for any $a\in \A$,
\begin{equation}
s_a=s_a^\perp
\end{equation}

and $(H_a\in \vect{H},a\in \A)$ is decomposable. 
\end{prop}
\begin{proof}

Let us note $\mu_\A$ as $\mu$; let us recall that by definition of $(S_a,a\in \A)$ for any $a\in \A$, 

\begin{equation}
\pi_as_b^\perp= 1[b\leq a] s_b^\perp
\end{equation}

Let $a\in \A$, let $v\in H$ and $w\in  S_a$,

\begin{equation}
\langle s_a(v),w\rangle= \sum_{b\leq a}\mu(a,b) \langle \pi_b s_a(v), w\rangle
\end{equation}

As for any $b\in \A$, $\pi_b$ is a projector,

\begin{equation}
\langle s_a(v),w\rangle=\sum_{b\leq a}\mu(a,b) \langle s_a(v), \pi_b w\rangle =\sum_{b\leq a}\mu(a,b)\delta_b(a) \langle s_a(v), \pi_b w
\end{equation}

and so as $\mu(a,a)=1$ for any $a\in \A$,

\begin{equation}
\langle v,w\rangle= \langle s_a(v),w\rangle
\end{equation}

Therefore $s_a=s_a^\perp$ which implies that $(H_a\in \vect{H},a\in \A)$ is decomposable.
\end{proof}

\begin{rem}
Proposition \ref{c-d:projector-decomposable-implies} implies that in this case for any $a\in \A$, $\im s_a$ are Hilbert spaces.
\end{rem}

The previous propositions (Propositions \ref{c-d:decomposable-implies-projector},\ref{c-d:projector-decomposable-implies}) and Theorem \ref{d-h:thm} and Theorem 3.1 \cite{GS3} show that for subspace of a given Hilbert space it is equivalent to prove that the collection of orthogonal projections $(\pi_a,a\in\A)$ is decomposable or that $(H_a,a\in \A)$ is decomposable; therefore one can use the intersection property for collection of projector or for collections of Hilbert spaces to verify if the presheaf given by the projectors is decomposable. Furthermore when $(\pi_a,a\in \A)$ is decomposable for any $\B\in \U(\A)$, 

\begin{equation}
\pi(\B)=\underset{b\in \B}{\sum}s_b=\underset{b\in \B}{\sum}s_b^{\perp}
\end{equation}

However if $(H_a,a\in \A)$ is not decomposable, $(s_a^\perp,a\in \A)$ and $(s_a,a\in \A)$ are usually different from one another. Indeed let $\A=\{0,0^{'},2\}$ such that $0\leq 1$, $0^{'}\leq 2$; let $H_1=H=\R\bigoplus \R$ with for any $(\lambda_1,\lambda_2),(\mu_1,\mu_2)\in H$, 

\begin{equation}
\langle(\lambda_1,\lambda_2),(\mu_1,\mu_2)\rangle=\underset{i=0,1}{\sum}\lambda_i\mu_i
\end{equation}

Let $H_0=\R\bigoplus 0$, let $H_{0^{'}}=\R e_1\oplus e_2$. One remarks that by construction,

\begin{equation}
\underset{a\in \A}{\sum}s_a=\id
\end{equation} 
but that,
\begin{equation}
\underset{a\in \A}{\sum}s_a^\perp= s_0^\perp+s_{0^{'}}^\perp\neq \id
\end{equation}

as $H_0$ and $H_{0^{'}}$ aren't othogonal to one another.

\section*{Acknowledgement}

 I am very grateful to Daniel Bennequin for our numerous discussions and his remarks on this paper. This work resulted from research supported by the University of Paris.

\appendix

\section{Results for the category of Pre-Hilbert spaces and Hilbert spaces}

\subsection{Completion}\label{a-h-c:Completion}
\begin{defn}\label{a-h-c:category-defn}

We shall consider several categories; the category that has as objects Hilbert spaces and as morphism continuous linear applications will be noted $\cat{Hilb}$, when the morphism are isometries it is $\cat{IHilb}$; when the objects are pre-Hilbert and the morphisms are continuous linear applications the corresponding category is $\cat{PHilb}$; finally the category that has as object pre-Hilbert spaces and as morphism isometries shall be denoted as $\cat{IPhilb}$ and the category that has as objects vector spaces and morphisms linear applications shall be noted as $\cat{Vect}$.
\end{defn}

\begin{prop}\label{a-h-c:completion}
Let $H$ be a pre-Hilbert space, there is a Hilbert space that we shall note as $\com H$ and a continuous injective morphism $\eta_H:H\to \com H$ such that for any Hilbert spaces $H_1$ and any continuous linear application (of $\cat{PHilb}$), $\phi: H\to H_1$, there is a unique continuous linear application $\com \phi:\com H\to H_1$ such that 
\begin{equation}
\com \phi \circ \eta_H= \phi
\end{equation}

i.e. such that the following diagram commutes,

\begin{equation}
\begin{tikzpicture}[baseline=(current  bounding  box.center),node distance=2cm, auto]
\node (A) {$H$ };
\node (D) [right of =A] {$H_1$};
\node (B) [below of=D] {$\com H$};
\draw[->] (A) to node {$\eta_H $} (B);
\draw[->] (B) to node [right]{$\com \phi$} (D);
\draw[->] (A) to node {$\phi$} (D);
\end{tikzpicture}
\end{equation}

$\com H$ is called the completion of $H$.

\end{prop}


\subsection{Epic and monic}\label{a-h-c:Epic and monic}

\begin{prop}\label{a-h-c:mono-epi}
Monomorphism of $\ph$ are injective continuous linear applications and epimorphism are surjective continuous linear applications.\\

Monomorphisms of $\h$ are injective applications and epimorphism are maps that have a dense image.\\

Monomorphisms of $\ih$ are injective applications and epimorphism are surjective applications. 
\end{prop}
\begin{proof}
Let $H,H_1$ be pre-Hilbert spaces, let $\phi:H\to H_1$ be monic, let $v,v_1\in H$ be such that $\phi(v)=\phi(v_1)$. Let,

\begin{equation}
\begin{array}{ccccc}
i& : &\mathbb{K} & \to & H\\
& & \lambda & \mapsto &\lambda v\\
\end{array}
\end{equation}

and,

\begin{equation}
\begin{array}{ccccc}
i_1& : &\mathbb{K} & \to & H\\
& & \lambda & \mapsto &\lambda v_1\\
\end{array}
\end{equation}

Then,

\begin{equation}
\phi i=\phi i_1
\end{equation}

and so,

\begin{equation}
i=i_1
\end{equation}

Therefore $v=i(1)=i_1(1)=v_1$.\\

The previous proof also holds in $\h$ as any finite dimentional normed vector space is complete, i.e. $i$ and $i_1$ are morphisms of $\h$ when $H$ is a Hilbert space.\\

To show that injective application are monomorphisms, let $H, H_1, H_2$ be Hilbert or pre-Hilbert spaces and let $\phi: H_1\to H_2$ be an injective application, let $\psi: H\to H_1$ and $\psi_1: H\to H_1$ be such that, 
\begin{equation}
\phi\psi=\phi\psi_1
\end{equation}

then for any $v\in H$, $\phi(\psi(v))= \phi(\psi_1(v))$ and as $\phi$ is injective, $\psi_1(v)=\psi(v)$.\\

Let us now show that if $\phi: H\to H_1$ is an epimorphism of pre-Hilbert spaces its image is dense in $H_2$. Let us note $K= \im \phi$; we shall distinguish between the closure of $K$ in $H_1$, that we shall note as $\overline{K}\subseteq H_1$, and the closure of $K$ in $\com H_1$, that we shall note as $\overline{K}\subseteq \com H_1$. Let us note $p$ the projection onto $\overline{K}\subseteq \com H_1$; then for any $v\in H_1$ and $w\in K$, 

\begin{equation}
\langle v, w\rangle = \langle p(v),w\rangle
\end{equation}

therefore $\langle v, .\rangle = \langle p(v),.\rangle$, i.e.,
\begin{equation}
\langle v,\phi\rangle= \langle p(v),\phi \rangle
\end{equation}

and so $v=p(v)$ in $\com H_1$, but as $v\in H_1$ and $p(v)\in \overline{K}\subseteq \com H_1$ then $v\in \overline{K}\subseteq H_1$.\\

When $H$ and $H_1$ are Hilbert spaces, $\langle v,.\rangle $ and $\langle p(v),.\rangle$ are morphism of $h$ therefore, epimorphism have a dense image.

Let $\phi: H\to H_1$ have a dense image in $H_1$, let $\psi,\psi_1: H_1\to H_2$ be two morphims such that,

\begin{equation}
\psi \phi=\psi_1\phi
\end{equation} 

Let $v\in H_1$, there is a sequence $(w_n\in H,n\in \N)$ such that $\lim_{n\to\infty}\phi(w_n) v$, therefore,

\begin{equation}
\psi(v)=\lim_{n\to \infty}\psi\phi(w_n)= \lim_{n\to \infty}\psi_1\phi(w_n)=\psi_1(v)
\end{equation}

This argument still holds for morphism of $\h$ and therefore epimorphisms of $\h$ are exactly the applications that have a dense image. \\

Following the same argument than for $\ph$ one can show that epimorphisms of $\ih$ are injective maps.\\

Let $\phi:H\to H_1$ be an epimorphism of $\ih$, then $\phi$ has a dense image; but as $H$ is a Hilbert space and as $\phi$ is an isometry, $\im \phi$ is a Hilbert space. In particular $\im \phi$ is closed and therefore $\im \phi =H_1$. Now when $\phi:H\to H_1$ is surjective it is a monomorphism.

\end{proof}

\begin{rem}
In general, for $\h$ and $\ph$, epic and monic does not imply isomorphism. However in $\ih$ it does.
\end{rem}

\begin{prop}
Let $\phi$ be a morphism of $\h$ such that $\phi$ is injective and surjective, then $\phi$ is an isomorphism.
\end{prop}
\begin{proof}
Let $\phi$ be injective and surjective, therefore by Banach-Schauder theorem it is an isomorphism as its inverse is continuous.

\end{proof}

\begin{rem}
There are morphisms that are epic and monic of $\ph$ but that are not isomorphisms; let $H= \bigoplus_{n\in \N} \mathbb{K}$, and let, $s:H\to H$ be such that for any $n\in \N$,

\begin{equation}
s(e_n)=\frac{1}{n} e_n
\end{equation}

then $s$ is injective and surjective but is not an isomorphism as its inverse $d:H\to H$, defined for any $n\in \N$ as,

\begin{equation}
d(e_n)=n e_n
\end{equation}

is not bounded.

\end{rem}

\subsection{Forgetful Functors and left adjoints}\label{a-h-c:Forgetful Functors and left adjoints}
\begin{defn}
Let $U:\cat{Hilb}\to \cat{Philb}$ be the forgetful functor that forgets the fact that a Hilbert space is complete, i.e. any Hilbert space is a pre-Hilbert space and any morphism of Hilbert spaces is a morphism of pre-Hilbert spaces.
\end{defn}

\begin{prop}\label{a-h-c:adjoint}
The left adjoint to $U$ is $\com$. Furthermore, $U\com U\com=U\com$; $U$ is fully faithfull and $\com$ is faithfull.
\end{prop}
\begin{proof}

$\eta: \id_\cat{PHilb}\to U\com $ is a natural transformation. Let for any Hilbert space, $H$, $\epsilon_H= \id$ then $\epsilon$ is a natural transformation from $\id_\cat{Hilb}\to \com U$ where $\com U$ cen be identified to $\id_{\cat{Hilb}}$, here we choose $\com U= \id_{\cat{Hilb}}$. One has that, 

\begin{equation}
U\star \epsilon \eta\star U= id_U
\end{equation}

and,

\begin{equation}
\epsilon \star \com \com \star \eta= id_{\com}
\end{equation}

Therefore $\com $ is left adjoint to $U$ (remark under Proposition A.5.7  Appendix A \cite{coend}). Furthermore as $\com U= id$, $U\com U\com= U\com$ and as $\epsilon$ is an isomorphism and $\eta$ is a monomorphism, one has that $U$ is fully faithfull and $\com$ is faithfull (Proposition A.5.9 \cite{coend}).
\end{proof}

\begin{cor}
The left adjoint to $U|_{\cat{IHilb}}^{\cat{PIhilb}}$ is $\com|_{\cat{PIhilb}}^{\cat{IHilb}}$. 
\end{cor}
\begin{proof}
$\eta$ and $\epsilon$ have their morphism in $\cat{IPhilb}$ and $\cat{IHilb}$, i.e. they are isometries.
\end{proof}

\begin{rem}
The forgethful functor $U_1:\ip\to \ph$ does not have left adjoint functor as, $0$ is final in $\cat{PHilb}$ but $0$ is not final in $\ip$ (there is no isometry form $\mathbb{K}\oplus \mathbb{K}$ to $0$).\\

 Furthermore if $U_1$ has a right adjoint, for any object $H$ of $\cat{PHilb}$, object $H_1$ of $\cat{IPhilb}$ and morphism of $\cat{Philb}$, $\phi:U_1 H_1\to H$, if there were $G(H)$ an object of $\cat{IPhilb}$ and $\epsilon_H: U_1GH\to H$ such that there was a unique isometry $\psi: H_1\to GH$ for which $\phi=\epsilon_H\psi $ then,
 
\begin{equation} 
\Vert\phi\Vert\leq \Vert\epsilon\Vert
\end{equation}

The last equation leads to a contradiction when choosing $\Vert \phi\Vert >\Vert \epsilon \Vert$.

\end{rem}

\begin{prop}\label{a-h-c:completion-closure}
Let $H$ be a Hilbert space and $H_1$ a pre-Hilbert subspace of $H$; 
\begin{equation}
\overline{H_1}\cong \com H_1
\end{equation} 
\end{prop}

\begin{proof}
Let $H_2$ be a Hilbert space and $\phi:H_1\to H_2$ be a continuous linear application. Let $(x_n\in H_1,n\in \N)$ be a sequence such that 
\begin{equation}
\lim_{n\to \infty}x_n= x
\end{equation}

 with $x\in H$; then $(x_n,n\in \N)$ is a Cauchy sequence, therefore $(\phi(x_n),n\in \N)$ is a Cauchy sequence and as $H_2$ is complete there is $y\in H_2$ such that,

\begin{equation}
\lim_{n\to n \infty}\phi(x_n)=y
\end{equation}
 
We shall say that $xRy$.\\

 For any other sequence $({x_1}_n,n\in N)$ such that,
 
\begin{equation}
\lim_{n\to \infty} {x_1}_n=x
\end{equation} 

 and for any $\epsilon >0$, there is $N\in \N$ such that 
\begin{equation}
\Vert y -y_1\Vert \leq \Vert y-\phi(x_n)\Vert +\Vert y_1-\phi({x_1}_n)\Vert+ \Vert\phi\Vert \Vert x_N-{x_1}_N\Vert <\epsilon
\end{equation}

Therefore the relation $R$ is a functional relation and it extends $\phi$ and it is the unique map to do so. 

\end{proof}

\subsection{Direct sum}\label{a-h-c:Direct sum}

We shall now recall and extend some results on the Direct sum in the catergory of Hilbert space (see for example \cite{Heunen}).\\

\begin{defn}
Let $U_2: \ph\to \cat{Vect}$ be the forgethful functor that forgets the scalar product of a pre-Hilbert space, i.e. any pre-Hilbert space is a vector space and any morphism of pre-Hilbert spaces is a vector space morphism.
\end{defn}

\begin{defn}\label{a-h-c:direct-sum}
Let $I$ be any set and $(H_i, \langle,\rangle_i,i \in I)$ be a collection of pre-Hilbert spaces. Let $\underset{i\in I}{\bigoplus}U_2H_i$ be the direct sum in $\cat{Vect}$. For any $u,v\in \underset{i\in I}{\bigoplus}U_2 H_i$ let,
\begin{equation}
\langle u,v\rangle=\underset{i\in I}{\sum}\langle u_i,v_i\rangle_i
\end{equation} 

$(\underset{i\in I}{\bigoplus}U_2H_i,\langle,\rangle)$ is a pre-Hilbert space and we shall note also note is as $\bigoplus_{i\in I}H_i$. We shall note its inclusions as $\psi_i$ and the projections as $\pi_i$.
\end{defn}

\begin{rem}
Let $(H_i,i\in I)$ be a collection of pre-Hilbert spaces, the inclusion $\psi_i: H_i\to \bigoplus_{i\in I}H_i$ is an isometry.
\end{rem}

\begin{prop}\label{a-h-c:no-colimit}
$\ph$, $\h$, $\ih$, $\ip$ are not cocomplete for non finite diagrams, i.e. if $D$ is a functor from any category $\cat{C}$ to one of these four categories then $\colim D$ does not necessarily exist. 

\end{prop}

\begin{proof}

Let $\cat{C}= \mathbb{N}$ seen as a set and not as a poset. Let $D$ be a functor from $\N$ to $\h$.\\

We shall show that there can't be a Hilbert space $(\colim D,(\eta_n,n\in \N))$ such that for any $H$ Hilbert space and $(\phi_n: D(n)\to H,n\in \N)$ collection of continuous linear applications, these morphisms factor through $\eta_n$, i.e. there is $\phi$ such that for any $n\in \N$, 

\begin{equation}
\phi_n =\phi \eta_n
\end{equation}

Let $D(n)=\mathbb{K}$, let us assume that $\colim D$ exists. Let us first show that for any $n\in \N$, $\eta_n$ is injective. Let $H=\com \bigoplus_{n\in \N} \mathbb{K}$, then as the inclusions $\psi_n$ are injective, one has that,

\begin{equation}
\Vert \eta_n\Vert \neq 0
\end{equation}

Let for any $n\in \N$, 
\begin{equation}
\begin{array}{ccccc}
\phi_n& : &\mathbb{K}& \to &\mathbb{K}\\
& & \lambda & \mapsto & n\lambda/\Vert \eta_n\Vert\\
\end{array}
\end{equation}

Then for any $n\in \N$,

\begin{equation}
\Vert \phi \Vert \geq n
\end{equation}

which contradictory. Therefore for $\h$, colimits over any category do not necessarily exist.\\

As $\com$ is left adjoint to $U$ the forgetful functor from $\h$ to $\ph$, one has that,

\begin{equation}
\colim D=\colim \com U D=\com \colim U D
\end{equation}

As $\colim D$ does not necessarily exist, $\colim U D$ does not necessarily exist which shows that colimits do not necessarily exist in $\ph$.\\

Let us now show that the same statement holds in $\ih$. Let $D$ be defined as previously and let us assume that $\colim D$ exists. For any $n\in \N$, $\psi_n: D(n)\to \com \bigoplus_{n\in \N} \mathbb{K}$ factor through $\eta_n$, i.e there is $\psi: \colim D\to \bigoplus_{n\in \N} \mathbb{K}$ such that for any $n\in \N$, 

\begin{equation}
\psi_n=\psi \eta_n
\end{equation}

Therefore for any $\lambda \in \bigoplus \mathbb{K}$, 
\begin{equation}
\psi(\sum_{n\in \N} \lambda_n \eta_n(1))= \bigoplus_{n\in \N} \lambda_n 
\end{equation}

And $\psi$ is an isometry, therefore,

\begin{equation}
\colim D\cong \com \bigoplus_{n\in \N}\mathbb{K}
\end{equation}

Let us now consider for any $n\in \N$,$\phi_n=\id$. Then $\phi$ can't be an isometry as $\phi(e_n\times n)=\phi(e_m\times m)$ for any $n,m\in \N$.\\

One concludes that in $\ih$ not every diagram $D$ has a colimit and, using the same argument that enabled us to conclude for $\ph$ from $h$, we conclude that $\ip$ does not have all colimits.

\end{proof}

\begin{rem}
In $\ih$, following the same lines than in Proposition \ref{a-h-c:no-colimit} but for two copies of $\mathbb{K}$, enables us to show there can be no finite sum, i.e. $\ih$ is not (finitely) cocomplete. 
\end{rem}

\subsection{Quotient of a pre-Hilbert space by a subspace}

\begin{rem}
Let $H$ be a pre-Hilbert space; the data of its scalar product is the same as the data of its norm by polarization identity, i.e. if $\mathbb{K}= \R$, for any $u,v\in H$,

\begin{equation}
\langle u,v\rangle= \frac{1}{4}\left( \Vert u+v\Vert^2 -\Vert u-v\Vert^2\right)
\end{equation}

and if $\mathbb{K}=\C$, 
\begin{equation}
\langle u,v\rangle= \frac{1}{4}\left( \Vert u+v\Vert^2 -\Vert u-v\Vert^2 +i\Vert u-iv\Vert^2 -i\Vert u+i v\Vert^2\right)
\end{equation}

We will note $(H,\langle,\rangle)$ when we will refer to its scalar product and $(H,\Vert .\Vert)$ when we refer to its norm.

\end{rem}

\begin{prop}\label{a-h-c:quotient-map}
Let $(H, \Vert.\Vert)$ be a pre-Hilbert space and $H_1$ be a closed subspace of $H_1$; let for any $v\in H$, 

\begin{equation}\label{a-h-c:quotient-norm}
\Vert v\Vert_0=\underset{y\in H_1}{\inf}\Vert v-y\Vert
\end{equation} 

Then $(H/H_1, \Vert.\Vert_0)$ is a pre-Hilbert space. Furthermore $\pi:H\to H/H_1$, the quotient map, is continuous.
\end{prop}


\begin{nota}
Let $H$ be a Hilbert space and $H_1$ a closed Hilbert subspace of $H$, when we will refer to $H/H_1$ as being a pre-Hilbert space we will be refering to its canonical norm $\Vert.\Vert_0$ (Proposition \ref{a-h-c:quotient-map}).
\end{nota}

\begin{lem}[Fréchet-Von Neumann-Jordan]\label{a-h-c:parallelogram}
Let $(B,\Vert .\Vert)$ be a $\mathbb{K}$-normed space, whose norm satifies the parallelogram identity, i.e for any $u,v\in B$,

\begin{equation}
\Vert u+v \Vert + \Vert u-v\Vert= 2(\Vert u\Vert + \Vert v\Vert)
\end{equation}

then $B$ is a pre-Hilbert space. 

\end{lem}

\begin{proof}
Refer to \cite{Jordan-vNeumann}
\end{proof}

\begin{lem}\label{a-h-c:sum-stable-approximation}

Let $H$ be a pre-Hilbert space and $H_1\subseteq H$ a pre-Hilbert subspace; let $x,y\in H$, let $\epsilon>0$, there is $\delta>0$ such that one can choose $v_1,v_2\in H_1$ such that,

\begin{equation}
\underset{u\in H_1}{\inf} \Vert x-u\Vert +\delta > \Vert x-v_1\Vert
\end{equation} 

\begin{equation}
\underset{u\in H_1}{\inf} \Vert y-u\Vert +\delta > \Vert y-v_2\Vert
\end{equation} 

And

\begin{equation}
\underset{u\in H_1}{\inf} \Vert x+y-u\Vert  +\epsilon >\Vert x+y-v_1-v_2\Vert
\end{equation} 

\end{lem}
\begin{proof}

We shall consider for this proof that the Hilbert space is a real Hilbert space, the proof for a complex Hilbert space follows the same line.\\

Let us consider $x\in H$, $u\in H_1$ ; let us note $k= u-v$. Let $\epsilon >0$ and let $u$ be such that,

\begin{equation}
\Vert x-u\Vert \leq \min_{v\in H_1}\Vert x-v\Vert +\epsilon 
\end{equation}

then, for any $w\in H_1$,

\begin{equation}
\Vert x-u\Vert \leq \Vert x-v\Vert =\epsilon
\end{equation}

One has that,

\begin{equation}
\Vert x-v\Vert^2= \Vert x-u\Vert^2 + 2\langle x-u,k\rangle +\Vert k\Vert^2
\end{equation}

Therefore,

\begin{equation}
2\langle x-u, k\rangle \leq \Vert k\vert^2+\epsilon
\end{equation}

The previous computations enable us to conclude that

\begin{equation}
\Vert x-u\Vert \leq \min_{v\in H_1}\Vert x-v\Vert +\epsilon 
\end{equation}

if and only if, for any $k\in H_1$,

\begin{equation}
2\langle x-u, k\rangle \leq \Vert k\Vert^2+\epsilon
\end{equation}

Furthermore if $u\in H_1$ is such that for any $k\in H_1$,

\begin{equation}
2\langle x-u, k\rangle \leq \Vert k\Vert^2+\epsilon
\end{equation}

then, for any $k\in H_1$ such that $k\neq 0$,

\begin{equation}
2\langle x-u, \frac{k}{\Vert k\Vert}\rangle \leq \Vert k\Vert +\epsilon
\end{equation}

and so for any $k\in H_1$ such that $\Vert k\Vert =1$,

\begin{equation}
2\langle x-u, k\rangle \leq \epsilon
\end{equation}

Furthermore if for any $k\in H_1$ such that $\Vert k\Vert =1$,

\begin{equation}
2\langle x-u, k\rangle \leq \epsilon
\end{equation}

then for any $k\in H_1$,

\begin{equation}
2\langle x-u, \frac{k}{\Vert k\Vert}\rangle \leq \Vert k\Vert \epsilon
\end{equation}

Assume that $\epsilon \leq 1$, then for $\Vert k\Vert \geq \epsilon$,
\begin{equation}
2\langle x-u, \frac{k}{\Vert k\Vert}\rangle \leq \Vert k\Vert^2\leq \Vert k\Vert^2 + \epsilon
\end{equation}

and when $\Vert k\Vert \leq \epsilon$,

\begin{equation}
2\langle x-u, \frac{k}{\Vert k\Vert}\rangle \leq \epsilon^2\leq \Vert k\Vert^2+\epsilon
\end{equation}

Therefore,

\begin{equation}
\Vert x-u\Vert \leq \min_{v\in H_1}\Vert x-v\Vert +\epsilon 
\end{equation}

if and only if for any $k\in H_1$ such that $\Vert k\Vert =1$,

\begin{equation}
2\langle x-u, \frac{k}{\Vert k\Vert}\rangle \leq \epsilon
\end{equation}

Let $1\geq \epsilon>0$, let $\delta=\epsilon/2$, let $y\in H$ and let $u, v \in H_1$ be such that,

\begin{equation}
\Vert x-u\Vert \leq \min_{v\in H_1}\Vert x-v\Vert +\delta
\end{equation}

and,

\begin{equation}
\Vert x-u\Vert \leq \min_{v\in H_1}\Vert x-v\Vert +\delta
\end{equation}

Then for any $k\in H_1$ such that 

\begin{equation}
2\langle x+y-u-v,k\rangle\leq \epsilon
\end{equation}

and,

\begin{equation}
\Vert x+y-u-v\Vert \leq \min_{v\in H_1}\Vert x+y-v\Vert +\epsilon
\end{equation}

\end{proof}

Let us now prove Proposition \ref{a-h-c:quotient-map},

\begin{proof}

Let $H/H_1$ be the quotient vector space of $H$ with respect to $H_1$, let for any $x\in H$, 

\begin{equation}
\Vert x\Vert_0=\underset{u\in H_1}{\min}\Vert x-u\Vert
\end{equation} 

For any $u\in H_1$, $\Vert x+u\Vert_0 =\Vert x\Vert_0$; therefore $\Vert .\Vert_0$ factorizes through the quotient map $[.]: H\to H/H_1$, i.e.for any $x\in H$,

\begin{equation}
\Vert [x]\Vert_0=\underset{v\in H_1}{\min}\Vert x-v\Vert
\end{equation}

is well defined.\\

Assume that $\Vert[x]\Vert_0=0$ then there is a sequence $(x_n\in H_1,n\in \N)$ such that $\lim_{n\to \infty}x_n=x$ and as $H_1$ is closed in $H$, one has that $x\in H$. Therefore for any $x\in H$, 

\begin{equation}
\Vert [x]\Vert_0=0 \implies [x]=0
\end{equation}

The norm defined by a scalar product is rigid in the sense that any sequence that optimises Equation \ref{a-h-c:quotient-norm} is unique with respect to the Cauchy equivalence (Lemma \ref{a-h-c:sum-stable-approximation}); therefore one can show that the parallelogram identity also holds for the quotient norm. In other words, Lemma \ref{a-h-c:sum-stable-approximation} unables us to assert that for any $x,y\in H$ and $\epsilon>0$, there is $u,v\in H_1$ such that,

\begin{equation}
\vert \Vert [x+y]\Vert_0 +\Vert [x-y]\Vert_0 -2(\Vert [x]\Vert_0 +\Vert [y]\Vert_0)\vert\leq \vert \Vert x-u+y-v\Vert +\Vert x-u-(y-v)\Vert -2(\Vert x-u\Vert +\Vert y-v\Vert)\vert +\epsilon 
\end{equation}

and therefore for any $\epsilon >0$,

\begin{equation}
\vert \Vert [x+y]\Vert +\Vert [x-y]\Vert -2(\Vert [x]\Vert +\Vert [y]\Vert)\vert\leq \epsilon 
\end{equation}

Lemma \ref{a-h-c:parallelogram} shows that $\Vert.\Vert_0$ is the norm associated to a scalar product, which end the proof.\\

Finally, for any $x\in H$,

\begin{equation}
\Vert [x]\Vert_0\leq \Vert x\Vert 
\end{equation}

Therefore $\pi$ is continuous.
\end{proof}

\begin{rem}
An other way to prove Proposition \ref{a-h-c:quotient-map} is to remark that if one completes $H$ and $H_1$, then if one notes $p$ the orthogonal projection on $H_1^\perp\subseteq \com H$, one has that for any $v\in H\subseteq \com H$,

\begin{equation}
\Vert p(v)\Vert =\underset{u\in \com H_1}{\min} \Vert v-u\Vert
\end{equation} 

Furthermore as $\Vert.\Vert$ is by definition continuous, one has that,
\begin{equation}
\underset{u\in \com H_1}{\inf} \Vert v-u\Vert=\underset{u\in H_1}{\inf} \Vert v-u\Vert
\end{equation}

Therefore $\Vert[v]\Vert^2= \langle p(v),p(v)\rangle$ which by polarisation defines a positive bilinear form (or sesquilinear form) on $H/H_1$; as $H_1$ is closed in $H$ it is also a definite bilinear form as for any $v\in H$, if $\langle(v),p(v)\rangle=0$ then $v$ in the closure of $H_1$ in $\com H$, but as $v\in H$ it must be in the closure of $H_1$ in $H$ which is $H$. Therefore it is a scalar product.\\

\end{rem}

\begin{prop}\label{a-h-c:coequalizer}

Let $(H,\Vert.\Vert)$ be a pre-Hilbert space and $H_1$ be a closed pre-Hilbert subspace of $H$; we shall note $i:H_1\to H$ the inclusion. The coequalizer $\coeq(i,0)$ exists and is equal to $H/H_1$.
\end{prop}

\begin{proof}
Let $\phi: H\to H_2$ be a morphism such that, 

\begin{equation}
\phi\circ i=0
\end{equation}

Then $\phi$ factors uniquely through, $\pi: H\to H/H_1$, i.e. there is a unique $\phi_1: H/H_1\to H_2$ such that, 

\begin{equation}
\phi_1\pi_=\phi
\end{equation} 

Let us now show that $\phi_1$ is continuous. For any $v\in H$,

\begin{equation}
\Vert\phi_1([v])\Vert \leq \Vert \phi\Vert \Vert v\Vert 
\end{equation}

therefore for any $u\in H_1$,

\begin{equation}
\Vert\phi_1([v])\Vert \leq \Vert \phi\Vert \Vert v+u\Vert 
\end{equation}

and,
\begin{equation}
\Vert\phi_1([v])\Vert \leq \Vert \phi\Vert \inf_{u\in H_1}\Vert v+u\Vert 
\end{equation}

Which show that $\phi_1$ is continuous and ends the proof.
\end{proof}

\begin{prop}\label{a-h-c:adjoint-colimit}
Let $F: \cat{C}\to \cat{D}$ be a left adjoint functor to $G:\cat{D}\to\cat{C}$; let $\mathbb{D}:D\to \cat{C}$ be a diagram in $\cat{C}$, $F(\colim \mathbb{D})=\colim F\mathbb{D}$.
\end{prop}

\begin{cor}
Let $H$ be a pre-Hilbert space and $H_1$ be a closed subspace of $H$, then

\begin{equation}
\com (H/H_1)= \com H/ \com H_1
\end{equation}
\end{cor}

\begin{proof}
By Proposition \ref{a-h-c:coequalizer} and Proposition \ref{a-h-c:adjoint-colimit}.

\end{proof}

\subsection{For $\ih$, the poset of subobjects is a complete join-lattice}\label{a-h-c:The poset of subobjects is a complet join-lattice}

\begin{rem}
Let $\cat{C}$ be any category and let $\phi_1:A\to B$, $\phi_2:C\to B$ be two monomorphism and $\phi:A\to B$ be a morphism. If there is $\psi:D\to B$ a monomorphism  and $\psi_1:A\to D$ and $\psi_2:B\to D$ two morphisms such that,

\begin{equation}
\psi \psi_1=\phi_1
\end{equation}

\begin{equation}
\psi \psi_2=\phi_2
\end{equation}

then,

\begin{equation}
\psi_1=\psi_2\phi
\end{equation}

Indeed as,
\begin{equation}
\psi \psi_2\phi=\phi_2\phi=\phi=\psi\psi_1
\end{equation}

and as $\psi$ is a monomorphism,

\begin{equation}
\psi_1=\psi_2\phi
\end{equation}
This justifies that in what follows we only talk about collections of monomorphism with a given codomain and not about functors.
\end{rem}

\begin{prop}\label{a-h-c:sup}
Let $(H_i\to H,i\in I)$ be a collection of monomorphism of $H$ for the category $\ih$; let us refer to $H_i\to H$ simply as $H_i$. \\

There is a cocone over the functor $(H_i,i\in I)$ in the category of monomorphism with codomain $H$ (of $\ph$), $\bigvee_{i\in I}H_i$, that is initial. In other words there is a monomorphism $\bigvee_{i\in I}H_i\to H$, and a collection of morphisms $(\eta_i:H_i\to \bigvee_{i\in I}H_i,i\in I)$, satisfying for any $i\in I$,

\begin{equation}
\bigvee_{i\in I} H_i \eta_i=H_i
\end{equation}

such that if a monomorphism $W\to H$ and a collection $(\phi_i:H_i\to W, i\in I)$ of morphisms satisfies,

\begin{equation}
W\phi_i=H_i
\end{equation}

then there is a unique morphisms $\bigvee_{i\in I}\phi_i$ such that,

\begin{equation}
\bigvee_{i\in I}\phi_i\eta_i=\phi_i
\end{equation}

i.e the following diagram commute,

\begin{equation}\label{a-h-c:sup-diagram}
\begin{tikzpicture}[baseline=(current  bounding  box.center),node distance=2cm, auto]
\node (A) {$H_i$ };
\node (B) [below of=A] {$\bigvee H_i$};
\node (C) [right of =A] {$H$};
\node (D) [below of =C] {$W$};
\draw[->] (A) to node [left] {$\eta_i $} (B);
\draw[->] (A) to node {$H_i$} (C);
\draw[->] (B) to node [below]{$\bigvee \phi_i$} (D);
\draw[->] (D) to node [right]{$W $} (C);
\draw[->]   (A) to node {$\phi_i $} (D);
\draw[->]   (B) [out=-90,in=180]to [rounded corners](3,-3)[out=0,in =0] to  node [right]{$\bigvee H_i $} (C)  ;
\end{tikzpicture}
\end{equation}

\end{prop}

\begin{proof}

Let $W\to H$ be a monomorphism and let $(\phi_i:H_i\to W, i\in I)$ be a collecton of morphisms that satisfies,

\begin{equation}
W\phi_i=H_i
\end{equation}

Let $\overline{\sum_{i\in I}H_i}$ be the closure of the pre-Hilbert space generated by the $(\im H_i,i\in I)$. Let,

\begin{equation}
\bigvee_{i\in I}H_i=\overline{\sum_{i\in I}H_i}\to H
\end{equation}

let for $i\in I$, $\eta_i:H_i\to \bigvee_{i\in I}H_i$ be $H_i$ with codomain restricted to $\overline{\sum_{i\in I}H_i}$; and similarly let $\bigvee_{i\in I}H_i\to H$ be the inclusion of $\overline{\sum_{i\in I}H_i}$ into $H$. Then there is a unique linear application $\bigvee_{i\in I}\phi$ such that,

\begin{equation}
\bigvee_{i\in I}\phi_i\eta_i=\phi_i
\end{equation}

Let us note $\bigvee_{i\in I}\phi_i$ as $\Phi$, then for any $v\in \bigoplus_{i\in I}H_i$, 

\begin{equation}
\Phi(\sum_{i\in I}H_i(v_i))=  \sum_{i\in I} \phi_i(v_i)
\end{equation}

and the image of $W$ is $\overline{\sum_{i\in I}H_i}$ and $\Phi$ is the inverse of $W|^{\bigvee_{i\in I}H_i}$. Therefore $\Phi$ is also an isomety which ends the proof.

\end{proof}

\begin{defn}\label{a-h-c:subobject}
Let $\cat{C}$ be a category and $A$ and object of $\cat{C}$. A subobject of $A$ is an isomorphism class of monomorphisms, where two monomorphims $i:B\to A$ and $j:B_1\to A$ are equivalent if and only if there is an isomorphims $\phi:B\to B_1$ such that $j\phi=i$.\\

One can define a preorder on monomorphisms as follows,

\begin{equation}
i\leq j\quad \iff \quad \exists \phi:B\to B_1, j\phi=i
\end{equation}

The preoder defines an order between classes of equivalence sets; we refer the collection of these classes, equipped  with the order, as the subobject poset $\Sub(A)$.

\end{defn}

\begin{rem}
Let $\cat{C}$ be any category, let $A\to C$, $B\to C$ be two monomorphism such that $[A]\leq [B]$. Then there is $\phi:A\to B$ such that, 

\begin{equation}
B\phi= A
\end{equation}

Let $\psi,\psi_1:D\to A$ be two morphisms such that,

\begin{equation}
\phi\psi=\phi\psi_1
\end{equation}   

then $B\phi\psi= B\phi \psi_1$ and, $A\psi =\psi_1$. So $\psi=\psi_1$, and this show that $\phi$ is a monomorphism.
\end{rem}

\begin{cor}\label{a-h-c:join-lattice-isometry}
Let $H$ be a Hilbert space, the "poset" $\Sub(H)$ of $\ih$ is a complete join-lattice, i.e. for any set $I$ and collection of objects $([H_i],i\in I)$ of $\Sub(H)$ there is a subobject $\vee_{i\in I} [H_i]$ such that for any object $J$ of $\Sub(H)$ such that for any $i\in I$, $[H_i]\leq J$, one has that,

\begin{equation}
\vee_{i\in I}[H_i]\leq J
\end{equation}  

\end{cor}
\begin{proof}

Let $J=[W]$ be a subobject of $H$ then, applying Proposition \ref{a-h-c:sup}, one gets that,

\begin{equation}
[\vee_{i\in I}H_i]\leq [W]
\end{equation} 

therefore, $\vee_{i\in I}[H_i]=[\vee_{i\in I}H_i]$.
\end{proof}

\bibliographystyle{ieeetr}
\bibliography{bintro}

\begin{thebibliography}{10}

\bibitem{Speed}
T.~P. Speed, ``A note on nearest-neighbour gibbs and markov probabilities,''
  {\em Sankhy\=a: The Indian Journal of Statistics, Series A}, 1979.

\bibitem{Lauritzen}
S.~L. Lauritzen, {\em Graphical Models}.
\newblock Oxford Science Publications, 1996.

\bibitem{Kellerer1964}
H.~G. Kellerer, ``Masstheoretische {M}arginalprobleme,'' {\em Mathematische
  Annalen}, vol.~153, no.~3, pp.~168--198, 1964.

\bibitem{Sinai}
Y.~G.~S. Sinai, {\em Theory of phase transitions: Rigourous results}.
\newblock Pergamon Press, 1982.

\bibitem{T.Levy}
T.~Lévy, ``Mesures gaussiennes et espaces de fock,'' 2003.

\bibitem{GS2}
G.~Sergeant-Perthuis, ``Intersection property and interaction decomposition.''
  arXiv:1904.09017, 2019.

\bibitem{GS3}
G.~Sergeant-Perthuis, ``Interaction decomposition for presheaves.''
  arXiv:2008.09029, 2020.

\bibitem{Lauritzen1}
S.~L. Lauritzen, ``Interaction models,'' in {\em Selected Works of Terry
  Speed}, ch.~10, pp.~91--94, Springer, 2012.

\bibitem{GS1}
G.~Sergeant-Perthuis, ``Bayesian/graphoid intersection property for
  factorisation models.'' arXiv:1903.06026v1, 2019.

\bibitem{Giry}
M.~Giry, ``A categorical approach to probability theory,'' in {\em Categorical
  aspects of topology and analysis}, vol.~915 of {\em Lecture notes in
  Mathematics}, pp.~68--85, Springer, 1982.

\bibitem{Lawvere}
B.~Lawvere, ``The category of probabilistic mappings,'' 1962.

\bibitem{GSDB}
D.~Bennequin, O.~Peltre, G.~Sergeant-Perthuis, and J.~P. Vigneaux, ``Extra-fine
  sheaves and interaction decompositions.'' arXiv:2009.12646, 2020.

\bibitem{GSthese}
G.~Sergeant-Perthuis, {\em Intersection property, interaction decomposition,
  regionalized optimization and applications.}
\newblock PhD thesis, Université de Paris, 2021.

\bibitem{coend}
F.~Loregian, ``Coend calculus.'' arXiv:1501.02503, 2019.

\bibitem{Heunen}
C.~Heunen, ``On the functor $l^2$,'' in {\em Computation, Logic, Games, and
  Quantum Foundations} (G.~Goos, J.~Hartmanis, and J.~v. Leeuwen, eds.),
  pp.~107--121, Springer, 2013.

\bibitem{Jordan-vNeumann}
P.~Jordan and J.~Von~Neumann, ``On inner products in linear, metric spaces,''
  {\em Annals of Mathematics}, 1935.

\end{thebibliography}





\end{document}